\documentclass[11pt]{amsart}
\usepackage{lmodern}
\usepackage{amsmath, amsthm, amssymb, amsfonts}
\usepackage[normalem]{ulem}
\usepackage{hyperref}
\usepackage{mathrsfs}
\usepackage{verbatim} 
\usepackage{longtable}
\usepackage{mathtools}

\usepackage{tikz}
\usetikzlibrary{decorations.pathmorphing}
\tikzset{snake it/.style={decorate, decoration=snake}}
\usepackage{caption}
\usepackage{tikz-cd}
\usetikzlibrary{arrows}

\theoremstyle{plain}
\newtheorem{thm}{Theorem}[section]
\newtheorem{cor}[thm]{Corollary}
\newtheorem{lem}[thm]{Lemma}
\newtheorem{prop}[thm]{Proposition}
\newtheorem{conj}[thm]{Conjecture}

\theoremstyle{definition}

\theoremstyle{remark}
\newtheorem{rmk}[thm]{Remark}

\newcommand{\BC}{{\mathbb{C}}}

\newcommand{\BQ}{{\mathbb{Q}}}

\newcommand{\BZ}{{\mathbb{Z}}}

\newcommand{\CA}{{\mathcal A}}

\newcommand{\CE}{{\mathcal E}}
\newcommand{\CF}{{\mathcal F}}

\newcommand{\CH}{{\mathcal H}}
\newcommand{\CI}{{\mathcal I}}

\newcommand{\CL}{{\mathcal L}}
\newcommand{\CM}{{\mathcal M}}

\newcommand{\CR}{{\mathcal R}}
\newcommand{\CS}{{\mathcal S}}
\newcommand{\CT}{{\mathcal T}}
\newcommand{\CU}{{\mathcal U}}

\newcommand{\CX}{{\mathcal X}}

\newcommand{\Fp}{{\mathfrak{p}}}

\newcommand{\FF}{{\mathfrak{F}}}

\DeclareFontFamily{OT1}{rsfs}{}
\DeclareFontShape{OT1}{rsfs}{n}{it}{<-> rsfs10}{}
\DeclareMathAlphabet{\curly}{OT1}{rsfs}{n}{it}

\newcommand{\Pic}{\mathop{\rm Pic}\nolimits}

\begin{document}
\title[On the Orlov conjecture]{On the Orlov conjecture for hyper-K\"ahler varieties via hyperholomorphic bundles}
\date{\today}

\author[D. Maulik]{Davesh Maulik}
\address{Massachusetts Institute of Technology}
\email{maulik@mit.edu}

\author[J. Shen]{Junliang Shen}
\address{Yale University}
\email{junliang.shen@yale.edu}

\author[Q. Yin]{Qizheng Yin}
\address{Peking University}
\email{qizheng@math.pku.edu.cn}

\begin{abstract}
We study Fourier transforms induced by Markman's projectively hyperholomorphic bundles on products of hyper-K\"ahler varieties of~$K3^{[n]}$-type. As applications, we prove the following. (a) Derived equivalent hyper-K\"ahler varieties of $K3^{[n]}$-type have isomorphic homological motives preserving the cup-product. (b) All smooth projective moduli spaces of stable sheaves on a given $K3$ surface have isomorphic homological motives preserving the cup-product. (c) Assuming the Franchetta properties for the self-products of polarized $K3$ surfaces, the isomorphisms in (b) can be lifted to Chow motives for $K3$ surfaces of Picard rank $1$. These results provide evidence for the Orlov conjecture and a conjecture of Fu--Vial.

\end{abstract}

\maketitle

\setcounter{tocdepth}{1} 

\tableofcontents
\setcounter{section}{-1}

\section{Introduction}

Throughout, we work over the complex numbers; Chow groups and motives are taken with rational coefficients.

\subsection{Orlov conjecture for hyper-K\"ahler varieties}\label{sec0.1}

For a nonsingular projective variety~$X$, the Orlov conjecture \cite{Orlov} predicts a mysterious connection between the bounded derived category of coherent sheaves $D^b(X)$ and the Chow motive $h(X)$. We say that nonsingular projective varieties $X,X'$ are derived equivalent if there is an isomorphism of bounded derived categories of coherent sheaves,
\[
D^b(X) \simeq D^b(X').
\]

\begin{conj}[Orlov conjecture \cite{Orlov}]\label{conj0}
Let $X,X'$ be nonsingular projective varieties which are derived equivalent. There is an isomorphism of Chow motives
    \[
    h(X) \simeq h(X').
    \]
\end{conj}

In this paper, we focus on hyper-K\"ahler varieties, for which Fu--Vial \cite[Conjecture 4.7]{FV} proposed a \emph{multiplicative} version of Conjecture \ref{conj0}.

\begin{conj}[Multiplicative Orlov conjecture \cite{FV}]\label{conj1}
Let $X,X'$ be derived equivalent hyper-K\"ahler varieties.\footnote{In fact, if one of the varieties is hyper-K\"ahler, the derived equivalence condition forces the other to be hyper-K\"ahler of the same dimension by \cite{HNW}.} There is an isomorphism of Chow motives which respects the cup-product,
\[
( h(X), \cup ) \simeq (h(X'), \cup).
\]
\end{conj}

Assume that $X,X'$ are of dimension $2n$. Conjecture \ref{conj1} predicts the existence of algebraic correspondences
\begin{equation}\label{Orlov_cycle}
[Z] \in \mathrm{CH}^{2n}(X \times X'), \quad [Z'] \in  \mathrm{CH}^{2n}(X' \times X)
\end{equation}
which are mutually inverse
\[
[Z']\circ[Z] = [\Delta_X] \in \mathrm{CH}^{2n}(X \times X), \quad [Z]\circ[Z'] = [\Delta_{X'}] \in \mathrm{CH}^{2n}(X'\times X')\]
and respect the cup-product. Here $\Delta_{(-)}$ is the diagonal class, $\circ$ stands for the composition of correspondences, and the cup-product is induced by the small diagonal class.

When $X$ is a $K3$ surface, Conjecture \ref{conj0} was proven in \cite{H_Motives} and Conjecture \ref{conj1} was proven in \cite[Corollary 2]{FV}. In higher dimension, very little evidence is known for Conjectures~\ref{conj0} and~\ref{conj1}. Even the following homological specialization of Conjecture \ref{conj1} is wide open.

\begin{conj}[Multiplicative Orlov conjecture for homological motives \cite{FV}]\label{conj2}
Let $X,X'$ be derived equivalent hyper-K\"ahler varieties. There is an isomorphism of homological motives which respects the cup-product,
\[
( h_{\mathrm{hom}}(X), \cup ) \simeq (h_{\mathrm{hom}}(X'), \cup).
\]
\end{conj}

More precisely, Conjecture \ref{conj2} predicts the existence of algebraic cycles (\ref{Orlov_cycle}) which induce mutually inverse cohomological correspondences
\[
[Z']\circ[Z] = [\Delta_X] \in H^{4n}(X \times X, \BQ), \quad [Z]\circ[Z'] = [\Delta_{X'}] \in H^{4n}(X'\times X', \BQ)\]
respecting the cup-product
\[
(H^*(X, \BQ), \cup) \xrightleftharpoons[~~~{[Z']}~~~]{~~~{[Z]}~~~} (H^*(X', \BQ), \cup).
\]

We discuss some evidence for Conjecture \ref{conj2}. Taelman \cite[Theorem D]{Taelman} showed that derived equivalent hyper-K\"ahler varieties $X,X'$ have isomorphic Hodge structures
\[
H^i(X, \BQ) \simeq H^i(X', \BQ);\]
therefore, assuming the Hodge conjecture, one obtains the algebraic cycles (\ref{Orlov_cycle}) which induce isomorphisms of the homological motives $h_{\mathrm{hom}}(X)\simeq h_{\mathrm{hom}}(X')$. However, these isomorphisms \emph{a priori} do not preserve the cup-product. On the other hand, Fu--Vial \cite[Proposition~4.8]{FV} showed that derived equivalent hyper-K\"ahler varieties $X,X'$ admit an isomorphism of cohomology rings with $\BC$-coefficients 
\[
(H^*(X,\BC),\cup) \simeq (H^*(X', \BC), \cup);
\]
but it is unclear if such an isomorphism can be lifted to cohomology with $\BQ$-coefficients preserving the Hodge structures. Finally, we note that in all the known examples, derived equivalent hyper-K\"ahler varieties are deformation equivalent; in particular they share isomorphic cohomology rings.

\subsection{Hyper-K\"ahler varieties of $K3^{[n]}$-type}

Although the conjectures in Section \ref{sec0.1} seem to be out of reach in general, we show in this paper that much evidence of these conjectures can be obtained for hyper-K\"ahler varieties of $K3^{[n]}$-type using the projectively hyperholomorphic bundles constructed recently by Markman \cite{Markman}.

Our first result shows that Conjecture \ref{conj2} concerning homological motives is completely settled for hyper-K\"ahler varieties of $K3^{[n]}$-type.

\begin{thm}\label{thm0.4}
    Assume that $X,X'$ are derived equivalent hyper-K\"ahler varieties of $K3^{[n]}$-type. Then there is an isomorphism of homological motives which respects the cup-product,
    \[
    (h_{\mathrm{hom}}(X), \cup) \simeq (h_{\mathrm{hom}}(X'), \cup).
    \]
\end{thm}

\begin{rmk}
    It is shown in Remark \ref{rmkfrob} that the isomorphisms we construct for Theorem~\ref{thm0.4} and the subsequent Theorems \ref{thm0.5}, \ref{thm0.7}(b), and \ref{thm0.8} send a point class to a point class. In view of~\cite[Proposition~2.11]{FV}, all these isomorphisms respect the Frobenius algebra structure defined in~\cite[Section~2.2]{FV}.
\end{rmk}

Next, we consider an interesting class of $K3^{[n]}$-type varieties given by the moduli space of stable sheaves on a $K3$ surface. Let $S$ be a projective $K3$ surface whose Mukai lattice is denoted $\widetilde{H}(S, \BZ)$. We consider a primitive Mukai vector $v \in \widetilde{H}(S, \BZ)$ with $v^2\geq 0$. For a~$v$-generic polarization, the moduli space $M_{S,v}$ of stable sheaves $\CE$ on $S$ with~$v(\CE) = v$ is a hyper-K\"ahler variety of $K3^{[n]}$-type of dimension $2n =  v^2+2$.\footnote{For different choices of $v$-generic polarizations, we get birational hyper-K\"ahler varieties. We are mainly concerned with Chow motives in this paper; by a result of Reiss \cite{Riess} the Chow motive of a hyper-K\"ahler variety, as an algebra object, is not sensitive to the birational model.} The following property of these hyper-K\"ahler varieties was proven by Beckmann \cite[Corollary 9.6]{Beck}.

\begin{thm}[Beckmann \cite{Beck}]\label{thm_Beck}
    Let $v$ be a primitive Mukai vector on a projective $K3$ surface~$S$.  If $X$ is a hyper-K\"ahler variety of $K3^{[n]}$-type which is derived equivalent to $M_{S,v}$, then $X$ is birational to $M_{S,v'}$, where $v'$ is a primitive Mukai vector on $S$ with $v^2={v'}^2$.
\end{thm}

In view of Theorem \ref{thm_Beck} and the birational invariance of Chow motives for hyper-K\"ahler varieties \cite{Riess}, it is natural to ask how the homological and the Chow motives of~$M_{S,v}, M_{S,v'}$ are related. The proof of Theorem \ref{thm0.4} also yields the following.

\begin{thm}\label{thm0.5}
Let $S$ be a projective $K3$ surface. Let $v, v' \in \widetilde{H}(S, \BZ)$ be primitive Mukai vectors with $v^2 = {v'}^2=2n-2 \geq 0$. Then there is an isomorphism of the homological motives of~$M_{S,v}$ and $M_{S,v'}$ which respects the cup-product,
\[
(h_{\mathrm{hom}}(M_{S,v}), \cup) \simeq (h_{\mathrm{hom}}(M_{S,v'}), \cup).
\]
\end{thm}

Finally, we discuss evidence for Chow motives. Our main result is that when $S$ is of Picard rank $1$, we are able to realize the isomorphism in Theorem~\ref{thm0.5} via \emph{generically defined cycles}, which allows us to lift the homological statement to Chow using the (conjectural) Franchetta properties for the self-products of polarized $K3$ surfaces. 

Assume $g\geq 2$. A $K3$ surface of genus $g$ is a primitively polarized $K3$ surface $(S,H)$ with~$H^2=2g-2$. Let $\CS_g \to \CX_g$ be the universal family over the moduli stack of $K3$ surfaces of genus $g$. We say that $K3$ surfaces of genus $g$ satisfy the \emph{$n$-th Franchetta property} if the generically defined cycles 
\[
\mathrm{GDCH}^*(S^n):= \mathrm{Im}\left( \mathrm{CH}^*({\CS^n_g}_{/\CX_g}) \to \mathrm{CH}^*(S^n) \right)
\]
on the $n$-th product of any fiber $S\subset \CS_g$ map injectively to the cohomology via the cycle class map, \emph{i.e.},
\[
\mathrm{cl}: \mathrm{GDCH}^*(S^n) \hookrightarrow H^{*}(S^n, \BQ).
\]
For $n=1$, the Franchetta property was conjectured by O'Grady \cite{OG}, known as the \emph{generalized Franchetta conjecture} for $K3$ surfaces, and has been verified for low genus cases using the global geometry of the universal family \cite{PSY,FL, Lu}. The Franchetta property was further conjectured to hold in \cite{FLV, FLV2} for any $n\geq 1$; see \cite[Conjecture 1.5]{LV}. 

\begin{thm}\label{thm0.7}
Let $(S,H)$ be a $K3$ surface of genus $g$ with $\Pic(S) = \BZ H$. Let $v,v' \in \widetilde{H}(S, \BZ)$ be primitive Mukai vectors with $v^2={v'}^2=2n-2\geq 0$. 
\begin{enumerate}
    \item[(a)] If the $2n$-th Franchetta property holds for $K3$ surfaces of genus $g$, then there is an isomorphism of Chow motives
    \[
    h(M_{S,v}) \simeq h(M_{S,v'}).
    \]
    \item[(b)]  If the $3n$-th Franchetta property holds for $K3$ surfaces of genus $g$, then there is an isomorphism of Chow motives which respects the cup-product,
    \[
    (h(M_{S,v}), \cup) \simeq (h(M_{S,v'}), \cup).
    \]
\end{enumerate}
\end{thm}

\subsection{Idea of the proof}\label{sec0.3}
We first discuss the main difficulty in proving the Orlov conjecture. Assume that $X, X'$ are derived equivalent varieties of dimension $d$. The Chern characters of the Fourier--Mukai kernels $\CU, \CU^{-1}$ yield correspondences
\begin{equation}\label{corr0}
\mathfrak{F} \in \mathrm{CH}^*(X\times X'), \quad \mathfrak{F}^{-1}  \in \mathrm{CH}^*(X' \times X)
\end{equation}
satisfying
\[
\mathfrak{F}^{-1}\circ\mathfrak{F} = [\Delta_X], \quad \mathfrak{F}\circ\mathfrak{F}^{-1} = [\Delta_{X'}].
\]
The classes (\ref{corr0}) are \emph{a priori} of \emph{mixed} degree in the Chow groups. In order to prove (the homological or the Chow version of) the Orlov conjecture, we need to construct algebraic cycles of \emph{pure} degree
\begin{equation}\label{cor00}
[Z] \in \mathrm{CH}^d(X \times X'), \quad [Z'] \in \mathrm{CH}^d(X' \times X)
\end{equation}
with 
\begin{equation}\label{hom-Chow}
[Z']\circ[Z] = [\Delta_X] , \quad [Z]\circ[Z'] = [\Delta_{X'}].
\end{equation}
Here the identities (\ref{hom-Chow}) hold in cohomology or Chow, depending on the version of the Orlov conjecture one would like to prove. In general, it is not clear how to modify the cycles in (\ref{corr0}) of mixed degree to the desired~$d$-dimensional cycles in (\ref{cor00}).

Now we consider hyper-K\"ahler varieties of $K3^{[n]}$-type. We first observe that derived equivalent varieties $X,X'$ of $K3^{[n]}$-type are connected Hodge-theoretically; see Proposition \ref{D=>HI}. Therefore, both Theorems \ref{thm0.4} and \ref{thm0.5} are deduced from the following theorem.

\begin{thm}\label{thm0.8}
Let $X, X'$ be hyper-K\"ahler varieties of $K3^{[n]}$-type. If there is a (rational) Hodge isometry $H^2(X, \BQ) \simeq H^2(X', \BQ)$, then there is an isomorphism of homological motives which respects the cup-product,
\[
(h_{\mathrm{hom}}(X), \cup) \simeq (h_{\mathrm{hom}}(X'), \cup).
\]
\end{thm}

The proofs of the main theorems rely on a Fourier transform package discussed in Section~\ref{sec1} for hyper-K\"ahler varieties of~$K3^{[n]}$-type, which in cohomology is largely a reformulation of results of Markman \cite{Markman}. We prove Theorems \ref{thm0.4}, \ref{thm0.5}, and \ref{thm0.8} in Section \ref{New_Sec2}.


In Section \ref{New_Sec4}, we show that the isomorphism in Theorem \ref{thm0.5} can be constructed from generically defined cycles when the $K3$ surface is of Picard rank $1$, thereby proving Theorem~\ref{thm0.7}. The key is an algebraic moduli approach to a recent construction of Zhang \cite{Zhang} regarding moduli spaces of stable sheaves, which we discuss in Section~\ref{sec2}.


\subsection{Acknowledgements}
We are grateful to Eyal Markman and Ruxuan Zhang for helpful discussions.

D.M.~was supported by a Simons Investigator Grant.
J.S.~was supported by the NSF grant DMS-2301474 and a Sloan Research Fellowship.

\section{Fourier transforms via hyperholomorphic bundles}\label{sec1}

Conjectures of Beauville and Voisin \cite{B2, Voi} predict that the Chow motive of a hyper-K\"ahler variety behaves analogously to that of an abelian variety. The purpose of Section \ref{sec1} is to make this analogy more explicit from the perspective of \emph{Fourier transforms}. Although most of the discussion in this section is conjectural, the construction introduced here plays a crucial role in the proofs of the main theorems. 

\subsection{Abelian varieties}\label{sec1.1}
We first briefly review the theory of Fourier transforms in the study of algebraic cycles on abelian varieties \cite{B,DM}. 

Let $A$ be a $g$-dimensional abelian variety, and let $A^\vee:= \mathrm{Pic}^0(A)$ be its dual. The normalized Poincar\'e line bundle $\CL$ on $A^\vee\times A$ induces a derived equivalence of coherent categories~\cite{Mukai}, which further induces algebraic correspondences, known as the Fourier transforms
\begin{equation*}
\mathfrak{F}:= \mathrm{ch}(\CL) \in \mathrm{CH}^*(A^\vee \times A), \quad \mathfrak{F}^{-1}:= \mathrm{ch}(\CL^\vee) \in \mathrm{CH}^*(A\times A^\vee).
\end{equation*}
We set
\[
\mathfrak{p}_k:= \mathfrak{F}_k\circ  \mathfrak{F}^{-1}_{2g-k} \in \mathrm{CH}^g(A\times A), \quad \mathfrak{p}^\vee_k:= \mathfrak{F}^{-1}_k\circ  \mathfrak{F}_{2g-k} \in \mathrm{CH}^g(A^\vee\times A^\vee),\quad 0\leq k \leq 2g,
\]
where $\mathfrak{F}_i \in\mathrm{CH}^i(A^\vee \times A)$ and $\mathfrak{F}^{-1}_j\in\mathrm{CH}^j(A \times A^\vee)$ are the degree $i$ and degree $j$ parts of $\mathfrak{F}$ and $\mathfrak{F}^{-1}$ respectively. Since $\mathfrak{F}, \mathfrak{F}^{-1}$ are obtained from Fourier--Mukai equivalences, we obtain decompositions of the diagonal classes
\begin{equation}\label{Ch_Ku_Ab0}
[\Delta_A] = \sum_{k=0}^{2g} \mathfrak{p}_k \in \mathrm{CH}^g(A \times A), \quad  [\Delta_{A^\vee}] = \sum_{k=0}^{2g} \mathfrak{p}^\vee_k \in \mathrm{CH}^g(A^\vee \times A^\vee).\end{equation}
A key observation of \cite{B,DM} is that the decompositions (\ref{Ch_Ku_Ab0}) yield \emph{Fourier-stable multiplicative} Chow--K\"unneth decompositions
\begin{align}
h(A) &= \bigoplus_{k=0}^{2g} h_k(A),\quad h_k(A)=(A, \mathfrak{p}_k,0),\label{decomp1}\\ \quad h(A^\vee) &= \bigoplus_{k=0}^{2g} h_k(A^\vee),\quad h_k(A^\vee)=(A^\vee, \mathfrak{p}^\vee_k,0). \label{decomp2}
\end{align}
More precisely, the decompositions (\ref{decomp1}) and (\ref{decomp2}) respect the cup-product. Furthermore, in the category of Chow motives, the graded algebra object
\[
\bigoplus_{k=0}^{2g} h_k(A), \quad \cup: h_i(A) \times h_{j}(A) \to h_{i+j}(A)
\]
given by the cup-product is isomorphic through the Fourier transform $\mathfrak{F}$ to the graded algebra object (with appropriate Tate-twists)
\[
\bigoplus_{k=0}^{2g} h_{2g-k}(A^\vee)(g - k), \quad \ast: h_{2g-i}(A^\vee)(g - i) \times h_{2g-j}(A^\vee)(g - j) \to h_{2g-i-j}(A^\vee)(g - i - j)
\]
given by the convolution product induced by the abelian group structure $+: A^\vee\times A^\vee \to A^\vee$.

This package relies heavily on special geometric ingredients of abelian varieties, including the normalized Poincar\'e line bundle and the group structure.

\subsection{Hyper-K\"ahler varieties}\label{sec1.2}
A Fourier transform package similar to the one in Section~\ref{sec1.1} for abelian varieties is expected to hold for hyper-K\"ahler varieties. We provide here an explicit proposal for varieties of~$K3^{[n]}$-type. Our proposal is built on the recently constructed projectively hyperholomorphic bundles of Markman \cite{Markman} which we briefly review as follows; see also \cite[Section 1]{MSYZ}. 

Consider a projective $K3$ surface $S_0$ with $\mathrm{Pic}(S_0)=\BZ H$. Assume that the Mukai vector
\[
v_0:=(r, mH, s) \in \widetilde{H}(S_0, \BZ), \quad \mathrm{gcd}(r,s)=1
\]
is isotropic (\emph{i.e.}~$v_0^2=0$), and that all $H$-stable sheaves on $S_0$ with Mukai vector $v_0$ are stable vector bundles. Let $M_{S_0,v_0}$ be the moduli space of such stable vector bundles. Then~$M_{S_0,v_0}$ is again a $K3$ surface with a universal rank $r$ bundle $\CU$ on $M_{S_0,v_0} \times S_0$. Conjugating the Bridgeland--King--Reid (BKR) correspondence \cite{BKR} yields a projectively hyperholomorphic vector bundle $\CU^{[n]}$ on $M_{S_0,v_0}^{[n]} \times S_0^{[n]}$ of rank
\[
\mathrm{rk}(\CU^{[n]}) = n!r^n,
\]
which induces a derived equivalence
\begin{equation*}
\Phi_{\CU^{[n]}}: D^b(M_{S_0,v_0}^{[n]}) \xrightarrow{\simeq} D^b(S_0^{[n]}).
\end{equation*}

Now let $X$ be any hyper-K\"ahler variety of $K3^{[n]}$-type. There always exists a variety $Y$ of~$K3^{[n]}$-type and a twisted vector bundle $(\CE,\alpha_\CE)$ on $Y\times X$, such that this twisted vector bundle is deformed from $\CU^{[n]}$ and induces a (twisted) derived equivalence
\begin{equation}\label{FM}
\Phi_{(\CE,\alpha_\CE)}: D^b(Y, \alpha_Y) \xrightarrow{\simeq} D^b(X, \alpha_X)
\end{equation}
with $\alpha_X, \alpha_Y$ certain Brauer classes on $X, Y$ respectively. We denote by $(\CE^{-1}, \alpha_{\CE^{-1}})$ the (shifted) twisted vector bundle given by the kernel of the inverse of (\ref{FM}). Note that even for fixed $X$, $S_0$, and $v_0$ (which determines the $K3$ surface $M_{S_0,v_0}$ uniquely), the choice of $Y$ and $(\CE, \alpha_\CE)$ may not be unique.

We consider the \emph{kappa class} (see \cite[Section 2]{BB} or \cite[Definition 1.3]{Markman}) associated with the twisted vector bundle $(\CE, \alpha_\CE)$. For the twisted vector bundle $\CE$, the vector bundle 
\[
\CE^{\otimes \mathrm{rank}(\CE)} \otimes \mathrm{det}(\CE)^{-1}
\]
is untwisted, and the kappa class
\[
\kappa(\CE) \in \mathrm{CH}^*(Y \times X)
\]
is defined to be the $\mathrm{rank}(\CE)$-th root of the Chern character of $\CE^{\otimes \mathrm{rank}(\CE)} \otimes \mathrm{det}(\CE)^{-1}$. We define~$\kappa(\CE^{-1})$ similarly; they yield Fourier transforms
\[
\mathfrak{F}:= \kappa(\CE)\sqrt{\mathrm{td}_{Y\times X}} \in \mathrm{CH}^*(Y \times X),\quad  \mathfrak{F}^{-1}:= \kappa(\CE^{-1})\sqrt{\mathrm{td}_{X\times Y}} \in \mathrm{CH}^*(X\times Y).
\]
Since $(\CE^{-1}, \alpha_{\CE^{-1}})$ is the kernel for the Fourier--Mukai inverse of (\ref{FM}), the relations
\[
[\Delta_X] = \mathfrak{F}\circ \mathfrak{F}^{-1} \in \mathrm{CH}^{2n}(X \times X), \quad [\Delta_Y] = \mathfrak{F}^{-1}\circ \mathfrak{F} \in \mathrm{CH}^{2n}(Y \times Y)
\]
are obtained from the Grothendieck--Riemann--Roch theorem. In other words, the Fourier transforms $\mathfrak{F}, \mathfrak{F}^{-1}$ induce mutually inverse isomorphisms of \emph{ungraded} Chow motives
\begin{equation} \label{iso_Chow_XY}
h(Y)\xrightleftharpoons[~~~\mathfrak{F}^{-1}~~~]{~~~\mathfrak{F}~~~} h(X), \quad  \mathfrak{F}\circ \mathfrak{F}^{-1} = \mathrm{id}_{h(X)}, \quad  \mathfrak{F}^{-1}\circ \mathfrak{F} = \mathrm{id}_{h(Y)}.
\end{equation}

\begin{rmk}
    For $X, Y$ as above, a morphism of Chow motives $h(X) \to h(Y)$ is induced by a correspondence in $\mathrm{CH}^{2n}(X\times Y)$. Here, for a morphism of ungraded Chow motives, we allow the correspondence to take value in the \emph{ungraded} Chow group $\mathrm{CH}^*(X\times Y)$. As we discussed in Section \ref{sec0.3}, a key difficulty in studying the interaction between Chow motives and derived categories is to capture the information of the grading from the na{\"\i}ve ungraded isomorphisms~(\ref{iso_Chow_XY}). In our setting, this requires careful control of the Fourier transforms associated with hyper-K\"ahler varieties of $K3^{[n]}$-type, which we discuss in the next section.\end{rmk}

\subsection{A Fourier transform package}\label{sec1.3}
Consider any $X,Y,\CE, \CE^{-1}, \mathfrak{F}, \mathfrak{F}^{-1}$ as in Section \ref{sec1.2}. We describe in this section a conjectural Fourier transform package. For integers $0 \leq i, j \leq 4n$, let $\mathfrak{F}_i \in\mathrm{CH}^i(Y \times X)$ and $\mathfrak{F}^{-1}_j\in\mathrm{CH}^j(X \times Y)$ denote the degree $i$ and degree $j$ parts of $\mathfrak{F}$ and~$\mathfrak{F}^{-1}$ respectively.

\begin{conj}[Fourier vanishing]\label{conj1.1}
The following hold for $\mathfrak{F}, \mathfrak{F}^{-1}$.
\begin{enumerate}
\item[(a)] For any integers $0\leq i,j < 2n$ we have
\[
\mathfrak{F}_{2i + 1} = 0\in\mathrm{CH}^{2i + 1}(Y \times X), \quad \mathfrak{F}^{-1}_{2j + 1} = 0\in\mathrm{CH}^{2j + 1}(X \times Y).
\]

\item[(b)] For any integers $0\leq i,j\leq 2n$ with $i+j \neq 2n$, we have
\[
\mathfrak{F}_{2i} \circ \mathfrak{F}^{-1}_{2j}= 0 \in \mathrm{CH}^{2i+2j-2n}(X \times X), \quad  \mathfrak{F}^{-1}_{2i} \circ \mathfrak{F}_{2j}= 0 \in \mathrm{CH}^{2i+2j-2n}(Y \times Y).\]
\end{enumerate}
\end{conj}

\begin{rmk}
    The analogous relations in Conjecture \ref{conj1.1}(b) for an abelian variety are a consequence of the compatibility between the ``multiplication by $N$'' map and the normalized Poincar\'e line bundle; \emph{c.f.}~\cite{DM}.
\end{rmk}

Exactly as in the abelian variety case, we set
\begin{equation}\label{projectors}
\mathfrak{p}_{2k}:= \mathfrak{F}_{2k}\circ  \mathfrak{F}^{-1}_{4n-2k} \in \mathrm{CH}^{2n}(X\times X), \quad \mathfrak{p}^\vee_{2k}:= \mathfrak{F}^{-1}_{2k}\circ  \mathfrak{F}_{4n-2k} \in \mathrm{CH}^{2n}(Y\times Y),\quad 0\leq k \leq 2n.
\end{equation}  
Conjecture \ref{conj1.1}(a) yields decompositions
\begin{equation}\label{CK}
[\Delta_X] = \sum_{k=0}^{2n} \mathfrak{p}_{2k} \in \mathrm{CH}^{2n}(X \times X), \quad  [\Delta_Y] = \sum_{k=0}^{2n} \mathfrak{p}^\vee_{2k} \in \mathrm{CH}^{2n}(Y \times Y)
\end{equation}
as in (\ref{Ch_Ku_Ab0}).

\begin{prop}\label{prop1.3}
    Assume that Conjecture \ref{conj1.1} holds. Then we have the following.
    \begin{enumerate}
        \item[(a)] (Motivic decompositions) The classes \eqref{projectors} form orthogonal projectors, \emph{i.e.},
\[
\Fp_{2i} \circ \Fp_{2j} = \begin{cases} \Fp_{2i} & i = j \\ 0 & i \neq j, \end{cases} \quad
\Fp^\vee_{2i} \circ \Fp^\vee_{2j} = \begin{cases} \Fp^\vee_{2i} & i = j \\ 0 & i \neq j. \end{cases}
\]
In particular, the identities in \eqref{CK} induce Chow--K\"unneth decompositions
\begin{align*} 
h(X) &= \bigoplus_{k=0}^{2n} h_{2k}(X),\quad h_{2k}(X)=(X, \mathfrak{p}_{2k},0), \\
\quad h(Y) &= \bigoplus_{k=0}^{2n} h_{2k}(Y),\quad h_{2k}(Y)=(Y, \mathfrak{p}^\vee_{2k},0). 
\end{align*}
\item[(b)] (Fourier stability) For integers $i, j, k$ with $0 \leq i,j \leq 2n$, the restriction--projection of~$\mathfrak{F}_{2k}$ to
\[
\mathfrak{F}_{2k}: h_{2i}(Y)(2n - 2k) \to h_{2j}(X)
\]
is nonzero only if $j = k = 2n - i$. Similarly, the restriction--projection of $\mathfrak{F}^{-1}_{2k}$ to
\[
\mathfrak{F}^{-1}_{2k}: h_{2i}(X) \to h_{2j}(Y)(2k - 2n)
\]
is nonzero only if $j = k = 2n - i$. In particular, \eqref{iso_Chow_XY} induces mutually inverse isomorphisms of Chow motives
\[
h_{4n-2k}(Y)(2n - 2k) \xrightleftharpoons[~~~\mathfrak{F}^{-1}_{4n-2k}~~~]{~~~\mathfrak{F}_{2k}~~~} h_{2k}(X).
\]
\end{enumerate}
\end{prop}

\begin{proof}
We first show that the $\mathfrak{p}_{2k}$ form orthogonal projectors. For any $k$, we have by (\ref{CK}) that
    \begin{gather*}
\Fp_{2k} = [\Delta_X] \circ \Fp_{2k} = \left(\sum_{i = 0}^{2n}\mathfrak{F}_{2i} \circ \mathfrak{F}^{-1}_{4n - 2i} \right) \circ \left(\mathfrak{F}_{2k} \circ \mathfrak{F}^{-1}_{4n - 2k} \right) \\ = \sum_{i =0}^{2n}\mathfrak{F}_{2i} \circ (\mathfrak{F}^{-1}_{4n - 2i} \circ \mathfrak{F}_{2k}) \circ \mathfrak{F}^{-1}_{4n - 2k} = \Fp_{2k} \circ \Fp_{2k}, 
\end{gather*}
where we have used the vanishing
\[
\mathfrak{F}^{-1}_{4n - 2i} \circ \mathfrak{F}_{2k} =0,\quad i\neq k
\]
obtained from Conjecture \ref{conj1.1}(b) in the last identity. The other desired relations
\[
\mathfrak{p}_{2i}\circ \mathfrak{p}_{2j}=0,\quad \mathfrak{p}_{2k}^\vee= \mathfrak{p}^\vee_{2k}\circ \mathfrak{p}_{2k}^\vee, \quad \mathfrak{p}^\vee_{2i}\circ \mathfrak{p}^\vee_{2j}=0, \quad i\neq j
\]
are obtained similarly. This proves (a) provided that the motivic decompositions of $h(X), h(Y)$ realize the K\"unneth decompositions in cohomology. We refer the cohomological statement to Proposition~\ref{prop1.8}(a) below.

By definition, the first restriction--projection in (b) is induced by the correspondence
\begin{equation} \label{eq:fs1}
\Fp_{2j} \circ \mathfrak{F}_{2k} \circ \Fp^\vee_{2i} = \mathfrak{F}_{2j} \circ (\mathfrak{F}^{-1}_{4n - 2j} \circ \mathfrak{F}_{2k}) \circ \Fp^\vee_{2i}.
\end{equation}
Applying Conjecture \ref{conj1.1}(b) to the right-hand side, we see that \eqref{eq:fs1} is nonzero only if $j = k$, in which case we find
\[
\mathfrak{F}_{2j} \circ \mathfrak{F}^{-1}_{4n - 2j} \circ \mathfrak{F}_{2j} \circ \Fp^\vee_{2i} = \mathfrak{F}_{2j} \circ \Fp^\vee_{4n - 2j} \circ \Fp^\vee_{2i}.
\]
Furthermore, the pairwise orthogonality of the $\Fp_{2i}^\vee$ from (a) (which again relies on Conjecture~\ref{conj1.1}) shows that \eqref{eq:fs1} is nonzero only if $j = k = 2n - i$. This proves the first claim of (b). The remaining of (b) follows similarly.
\end{proof}


In order to study the interaction between the cup-product and the Fourier transform, we need to normalize $\mathfrak{F}, \mathfrak{F}^{-1}$ as follows; this is crucial in view of Proposition \ref{prop1.8}(b) below.

Recall that the rank of $\CE$ is $n!r^n$. We set the \emph{normalized} Fourier transforms as 
\begin{equation}\label{normal}
\begin{gathered}
\widetilde{\mathfrak{F}} = \sum_{k} \widetilde{\mathfrak{F}}_{2k}, \quad \widetilde{\mathfrak{F}}_{2k}: = r^{k-n}\mathfrak{F}_{2k}, \\
\widetilde{\mathfrak{F}}^{-1} = \sum_{k} \widetilde{\mathfrak{F}}^{-1}_{2k}, \quad \widetilde{\mathfrak{F}}^{-1}_{2k}:= r^{k-n}\mathfrak{F}_{2k}.
\end{gathered}
\end{equation}
Clearly, under Conjecture \ref{conj1.1}(a), they still induce isomorphisms of ungraded Chow motives 
\[
h(Y)\xrightleftharpoons[~~~\widetilde{\mathfrak{F}}^{-1}~~~]{~~~\widetilde{\mathfrak{F}}~~~} h(X),\quad
\widetilde{\mathfrak{F}}\circ \widetilde{\mathfrak{F}}^{-1} = \mathrm{id}_{h(X)}, \quad  \widetilde{\mathfrak{F}}^{-1}\circ \widetilde{\mathfrak{F}} = \mathrm{id}_{h(Y)},
\]
and the statements of Proposition \ref{prop1.3} hold identically if we replace $\mathfrak{F},\mathfrak{F}^{-1}$ by $\widetilde{\mathfrak{F}}, \widetilde{\mathfrak{F}}^{-1}$.

Following \cite[Section 10]{Markman}, we consider the convolution class
\[
\mathfrak{C}:=  \widetilde{\mathfrak{F}}^{-1} \circ [\Delta_X^{\mathrm{sm}}] \circ \left(\widetilde{\mathfrak{F}}\times\widetilde{\mathfrak{F}}\right) \in \mathrm{CH}^*(Y\times Y \times Y),
\]
and its induced \emph{ungraded} morphism
\begin{equation*}\label{convolution}
\ast: h(Y) \times h(Y) \to h(Y).
\end{equation*}
Here $[\Delta^{\mathrm{sm}}_X]$ is the class of the small diagonal $X\subset X \times X \times X$. The class $\mathfrak{C}$ is \emph{a priori} of mixed degree and depends on the variety $X$ and the twisted vector bundle $\CE$ on $Y \times X$. We conjecture that it is in fact of pure degree and intrinsic to $Y$.

\begin{conj}[Convolution]\label{conj1.4} The following hold for the convolution class $\mathfrak{C}$.
    \begin{enumerate}
        \item[(a)] (Purity) We have
        \[
        \mathfrak{C} \in \mathrm{CH}^{2n}(Y\times Y \times Y).
        \]
        \item[(b)] (Independence) The convolution class $\mathfrak{C}$ is independent of $X$ and $\CE$, and is therefore intrinsic to~$Y$.
    \end{enumerate}
\end{conj}

\begin{rmk}
\begin{enumerate}
\item[(a)] The parallel purity result for abelian varieties $A, A^\vee$ follows directly from the fact that $\mathfrak{C}$ is the class of the graph of $+: A^\vee \times A^\vee \to A^\vee$.
\item[(b)] Under Conjecture \ref{conj1.1}, it is clear that the normalization (\ref{normal}) of the Fourier transforms only affects Conjecture~\ref{conj1.4}(b).
\end{enumerate}
\end{rmk}

\begin{prop}\label{prop1.6}
    Assume that Conjectures \ref{conj1.1} and \ref{conj1.4} hold. Then we have the following.
    \begin{enumerate}
    \item[(a)] The cup-product on $h(X)$ respects the decomposition in Proposition \ref{prop1.3}(a), \emph{i.e.},
        \[
       \cup:  h_{2i}(X) \times h_{2j}(X) \to h_{2i+2j}(X).
        \]
    \item[(b)] The convolution product on $h(Y)$ (with appropriate Tate-twists) respects the decomposition in Proposition \ref{prop1.3}(a), \emph{i.e.},
    \[
    \ast: h_{4n - 2i}(Y)(2n - 2i) \times h_{4n - 2j}(Y)(2n - 2j) \to h_{4n - 2i - 2j}(Y)(2n - 2i - 2j).
    \]
    \item[(c)] The graded algebra object 
\[
\bigoplus_{k=0}^{2n} h_{2k}(X), \quad \cup: h_{2i}(X) \times h_{2j}(X) \to h_{2i+2j}(X)
\]
given by (a) is isomorphic through the Fourier transform $\widetilde{\mathfrak{F}}$ to the graded algebra object
\begin{gather*}
\bigoplus_{k=0}^{2n} h_{4n-2k}(Y)(2n - 2k), \\ \ast: h_{4n-2i}(Y)(2n - 2i) \times h_{4n-2j}(Y)(2n - 2j) \to h_{4n-2i-2j}(Y)(2n - 2i - 2j)
\end{gather*}
given by (b).
\end{enumerate}
\end{prop}

\begin{proof}
    We first prove (a). The desired statement is equivalent to the vanishing
    \begin{equation} \label{eq:conv1}
\mathfrak{p}_{2k}\circ [\Delta^{\mathrm{sm}}_{X}] \circ (\mathfrak{p}_{2i} \times \mathfrak{p}_{2j})=0, \quad i+j \neq k.
    \end{equation}
    Expanding the left-hand side of \eqref{eq:conv1}, we find
    \[
    \widetilde{\mathfrak{F}}_{2k}\circ \widetilde{\mathfrak{F}}_{4n-2k}^{-1} \circ \left(\widetilde{\FF}\circ \mathfrak{C} \circ(\widetilde{\FF}^{-1} \times \widetilde{\FF}^{-1}) \right)\circ (\widetilde{\mathfrak{F}}_{2i}\times  \widetilde{\mathfrak{F}}_{2j})\circ (\widetilde{\mathfrak{F}}^{-1}_{4n-2i} \times \widetilde{\mathfrak{F}}^{-1}_{4n-2j} ).
    \]
    Expanding further, only one terms survives the Fourier vanishing of Conjecture \ref{conj1.1}, namely
    \begin{equation} \label{eq:conv2}
    \widetilde{\mathfrak{F}}_{2k}\circ \widetilde{\mathfrak{F}}_{4n-2k}^{-1} \circ \widetilde{\FF}_{2k}\circ \mathfrak{C} \circ(\widetilde{\FF}^{-1}_{4n - 2i} \times \widetilde{\FF}_{4n - 2j}^{-1})\circ (\widetilde{\mathfrak{F}}_{2i}\times  \widetilde{\mathfrak{F}}_{2j})\circ (\widetilde{\mathfrak{F}}^{-1}_{4n-2i} \times \widetilde{\mathfrak{F}}^{-1}_{4n-2j} ).
    \end{equation}
    Notice that the left-hand side of \eqref{eq:conv1} is a class in $\mathrm{CH}^{4n}(X \times X \times X)$, while by Conjecture~\ref{conj1.4}(a), the class \eqref{eq:conv2} is in
    \[
    \mathrm{CH}^{4n - 2i - 2j + 2k}(X \times X \times X), \quad i + j \neq k.
    \]
    The degree difference forces the vanishing \eqref{eq:conv1}.

    To see (b, c), we note that the convolution class induces a commutative diagram of ungraded Chow motives
    \begin{equation} \label{eq:fouriercomp}
\begin{tikzcd}
h(Y) \otimes h(Y) \arrow[r, "\ast"] \arrow[d, "\widetilde{\mathfrak{F}} \otimes \widetilde{\mathfrak{F}}"] & h(Y) \arrow[d, "\widetilde{\mathfrak{F}}"] \\
h(X) \otimes h(X) \arrow[r, "\cup"] & h(X).
\end{tikzcd}
\end{equation}
By Proposition \ref{prop1.3}(b), the vertical arrows respect the decompositions of Proposition \ref{prop1.3}(a) and they admit inverses given by $\widetilde{\mathfrak{F}}^{-1}$. By (a), the horizontal bottom arrow also respects the decompositions of Proposition \ref{prop1.3}(a). Consequently, the horizontal top arrow also respects the decompositions, which confirms (b). In particular, (\ref{eq:fouriercomp}) induces commutative diagrams 
\begin{equation*}
\begin{tikzcd}
h_{4n - 2i}(Y)(2n - 2i) \otimes h_{4n - 2j}(Y)(2n - 2j) \arrow[r, "\ast"] \arrow[d, "\widetilde{\mathfrak{F}}_{2i} \otimes \widetilde{\mathfrak{F}}_{2j}"] & h_{4n - 2i - 2j}(Y)(2n - 2i - 2j) \arrow[d, "\widetilde{\mathfrak{F}}_{2i + 2j}"] \\
h_{2i}(X) \otimes h_{2j}(X) \arrow[r, "\cup"] & h_{2i+2j}(X)
\end{tikzcd}
\end{equation*}
with the vertical arrows isomorphisms of Chow motives. This proves (c).
\end{proof}

\subsection{Results in cohomology}\label{sec1.4}
We first note that both conjectures in Section \ref{sec1.3}, which concern algebraic cycles, hold cohomologically by Markman \cite{Markman}.


\begin{prop}[Markman \cite{Markman}]\label{prop1.8}
    With the notation as in Section \ref{sec1.3}, the following hold.
    \begin{enumerate}
    \item[(a)] The class $\FF$ in cohomology induces a degree-reversing morphism
    \[
    \FF: H^*(Y, \BQ) \to H^*(X, \BQ), \quad H^i(Y, \BQ) \to H^{4n - i}(X, \BQ).
    \]
    In particular, for any integers $0 \leq i,j < 2n$ we have
    \[
    \FF_{2i + 1} = 0 \in H^{4i + 2}(Y \times X, \BQ), \quad \FF^{-1}_{2j + 1} = 0 \in H^{4j + 2}(X \times Y, \BQ).
    \]
    For any integers $0 \leq i,j \leq 2n$ with $i+j\neq 2n$, we have
    \[
    \mathfrak{F}_{2i} \circ \mathfrak{F}^{-1}_{2j}= 0 \in H^{4i+4j-4n}(X \times X, \BQ), \quad  \mathfrak{F}^{-1}_{2i} \circ \mathfrak{F}_{2j}= 0 \in H^{4i+4j-4n}(Y \times Y, \BQ).
    \]
    We further have
    \[
    \mathrm{Im}(\Fp_{2k}) = H^{2k}(X, \BQ), \quad  \mathrm{Im}(\Fp^\vee_{2k}) = H^{2k}(Y, \BQ), \quad 0 \leq k \leq 2n. 
    \]
\item[(b)] The class $\mathfrak{C}$ in cohomology induces a convolution product
\[
\ast: H^{4n - 2i}(Y, \BQ) \times H^{4n - 2j}(Y, \BQ) \to H^{4n - 2i - 2j}(Y, \BQ)
\]
which is intrinsic to $Y$.
    \end{enumerate}
\end{prop}

\begin{proof}
The degree-reversing property of $\FF$ in (a) is \cite[Lemma 4.1(1)]{Markman}. For the Fourier vanishing, we view $\FF_{2i + 1}$ as a morphism
\begin{equation} \label{eq:ffodd}
\FF_{2i + 1}: H^*(Y, \BQ) \to H^*(X, \BQ)
\end{equation}
and $\mathfrak{F}_{2i} \circ \mathfrak{F}^{-1}_{2j}$ as the composition
\begin{equation}\label{comp}
H^*(X, \BQ) \xrightarrow{\mathfrak{F}_{2j}^{-1}} H^*(Y, \BQ) \xrightarrow{\mathfrak{F}_{2i}} H^*(X, \BQ).
\end{equation}
The vanishing of $\FF_{2i + 1}$ and $\mathfrak{F}_{2i} \circ \mathfrak{F}^{-1}_{2j}$ in cohomology is equivalent to the vanishing of \eqref{eq:ffodd} and~\eqref{comp}, respectively. Both are immediate consequences of the degree-reversing property of~$\FF$, together with the fact that $K3^{[n]}$-type varieties have no odd cohomology; see \cite[Corollary~4.2]{Markman}. The other half of the Fourier vanishing is parallel. The statement on the images of~$\Fp_{2k}, \Fp^\vee_{2k}$ also follows directly from the degree-reversing property.

For (b), Markman introduced in \cite[Section 10]{Markman} a cohomological convolution product for any hyper-K\"ahler variety, and showed that if the hyper-K\"ahler variety is of $K3^{[n]}$-type, this convolution product can be realized as the convolution product associated with the Fourier transform of Section \ref{sec1.3}. In particular, the normalization in (\ref{normal}) by the factor $r^{k - n}$ is crucial for the convolution product associated with the Fourier transform to be matched on the nose with Markman's intrinsic one; see \cite[Corollary 1.10]{Markman}.
\end{proof}

As a consequence, both Propositions in Section \ref{sec1.3} hold at the level of homological motives. We thus obtain K\"unneth decompositions of the homological motives
\[
h_{\mathrm{hom}}(X) = \bigoplus_{k = 0}^{2n}h_{\mathrm{hom}, 2k}(X), \quad h_{\mathrm{hom}}(Y) = \bigoplus_{k = 0}^{2n}h_{\mathrm{hom}, 2k}(Y),
\]
and the two graded algebra objects
\[
\left(\bigoplus_{k = 0}^{2n}h_{\mathrm{hom}, 2k}(X), \cup\right), \quad \left(\bigoplus_{k = 0}^{2n}h_{\mathrm{hom}, 4n-2k}(Y)(2n-2k), \ast\right)
\]
are isomorphic through the Fourier transform $\widetilde{\FF}$.

Next, we give a description of the kappa class of Section \ref{sec1.2} in terms of the (rational) \emph{extended Mukai lattice}. For a hyper-K\"ahler manifold $X$, its extended Mukai lattice is the~$\BQ$-vector space
\begin{equation*}\label{extended_Mukai}
\widetilde{H}(X,\BQ) = \BQ \alpha \oplus H^2(X, \BQ) \oplus \BQ\beta, \quad 
\alpha^2=\beta^2=0, \quad (\alpha,\beta) = -1,
\end{equation*}
where the pairing on $H^2(X,\BQ)$ is induced by the Beauville--Bogomolov--Fujiki (BBF) form, and $\alpha,\beta$ are orthogonal to $H^2(X,\BQ)$. By \cite{Taelman, Beck, Markman}, the lattice $\widetilde{H}(X, \BQ)$ encodes important information of its derived category. Let $\widetilde{{H}}(X, \BC):= \widetilde{H}(X, \BQ) \otimes_\BQ \BC$ be the extended Mukai lattice with $\BC$-coefficients.

We consider a derived equivalence (\ref{FM}) obtained from a $K3$ surface $S_0$ of Picard rank $1$ and an isotropic Mukai vector $v_0$. We first explain that an isometry 
\begin{equation}\label{varphi}
\varphi: \widetilde{H}(Y, \BC) \xrightarrow{\simeq}  \widetilde{H}(X, \BC)
\end{equation}
yields an isomorphism of (ungraded) $\BC$-vector spaces 
\[
\varphi^H: H^*(Y, \BC) \xrightarrow{\simeq} H^*(X, \BC).
\]
By the construction, $X,Y$ are deformed from $S_0^{[n]}, M^{[n]}_{S_0,v_0}$ respectively. We further deform the $K3$ surface $S_0$ to a special one $S'$ such that the corresponding moduli space~$M_{S',v'}$ is isomorphic to $S'$ itself. We have parallel transports
\begin{equation}\label{PTX}
\mathrm{PT}_X: \widetilde{H}(X, \BQ) \to \widetilde{H}(S'^{[n]}, \BQ), \quad \mathrm{PT}^H_X: H^*(X, \BQ) \to H^*(S'^{[n]}, \BQ), 
\end{equation}
and
\begin{equation}\label{PTY'}
\mathrm{PT}_Y: \widetilde{H}(Y, \BQ) \to \widetilde{H}(M^{[n]}_{S',v'}, \BQ), \quad \mathrm{PT}^H_X: H^*(Y, \BQ) \to H^*(M^{[n]}_{S',v'}, \BQ).
\end{equation}
Through a choice of an isomorphism 
\begin{equation}\label{iso}
M_{S',v'} \simeq S',
\end{equation}
the parallel transports (\ref{PTY'}) can further be expressed as
\begin{equation}\label{PTY}
\mathrm{PT}_Y: \widetilde{H}(Y, \BQ) \to \widetilde{H}(S'^{[n]}, \BQ), \quad \mathrm{PT}^H_X: H^*(Y, \BQ) \to H^*(S'^{[n]}, \BQ).
\end{equation}

Now, by the parallel transports (\ref{PTX}) and (\ref{PTY}), the isometry (\ref{varphi}) induces an isometry 
\begin{equation}\label{isometry5}
\mathrm{PT}_Y\circ \varphi \circ \mathrm{PT}^{-1}_X:  \widetilde{H}(S'^{[n]}, \BQ) \to  \widetilde{H}(S'^{[n]}, \BQ).
\end{equation}
Integration of the Looijenga--Lunts--Verbitsky (LLV) algebra provides an action of the special orthogonal group $\mathrm{SO}(\widetilde{H}(S'^{[n]}, \BC))$ on the total cohomology $H^*(S'^{[n]}, \BC)$,
\begin{equation}\label{integ}
\mathrm{SO}(\widetilde{H}(S'^{[n]}, \BC)) \to \mathrm{GL}( H^*(S'^{[n]}, \BC)).
\end{equation}
A further choice of a monodromy reflection of $S'$, which induces a monodromy reflection of~$S'^{[n]}$, upgrades \eqref{integ} to an action of the full orthogonal group
\begin{equation} \label{integ2}
O(\widetilde{H}(S'^{[n]}, \BC)) \to \mathrm{GL}( H^*(S'^{[n]}, \BC));
\end{equation}
see \cite[Section 10.1]{Markman}. By composing with the parallel transports (\ref{PTX}) and (\ref{PTY}), we obtain
\[
\varphi^H : H^*(Y, \BC) \xrightarrow{\mathrm{PT}^H_Y} H^*(S'^{[n]}, \BC) \to H^*(S'^{[n]}, \BC) \xrightarrow{(\mathrm{PT}_X^{H})^{-1}} H^*(X, \BC)
\]
where the second map is given by the integration map \eqref{integ2} applied to (\ref{isometry5}). This gives the desired (ungraded) assignment
\[
\left( \varphi: \widetilde{H}(Y, \BC) \xrightarrow{\simeq} \widetilde{H}(X, \BC) \right) \mapsto  \left( \varphi^H: H^*(Y, \BC) \xrightarrow{\simeq} H^*(X, \BC) \right).
\]

For the derived equivalence (\ref{FM}), recall that we have the Fourier transform
\[
\mathfrak{F}: H^*(Y, \BQ) \to H^*(X, \BQ)
\]
induced by the kappa class of the Fourier--Mukai kernel $\CE$. By results of Taelman \cite[Theorems~4.10 and 4.11]{Taelman}, the morphism $\mathfrak{F}$ induces an isometry of the extended Mukai lattice
\begin{equation} \label{eq:htilde}
\widetilde{H}(\mathfrak{F}): \widetilde{H}(Y, \BC) \to \widetilde{H}(X, \BC)
\end{equation}
whose construction is representation-theoretic; see also \cite[Proposition 3.1]{Markman}. Note that when~$n$ is even, the isometry $\widetilde{H}(\mathfrak{F})$ is defined up to a sign depending on the orientations of~$X, Y$.\footnote{We refer to \cite{Markman}, especially Section 3 and Convention 5.16 therein for a careful analysis of the sign. As the precise sign is irrelevant for our purpose, we write ``$\pm$'' in the sequel to simplify the exposition.}.

The following lemma due to Markman \cite{Markman} shows that the morphism $\mathfrak{F}$ is determined up to a sign by the isometry $\widetilde{H}(\mathfrak{F})$ via the integration map (\ref{integ2}).

\begin{lem}[\cite{Markman}]\label{lem1.10}
Set $\varphi$ to be the isometry $\widetilde{H}(\mathfrak{F})$. Then we have $\mathfrak{F} = (\pm \varphi)^H$.
\end{lem}

\begin{proof}
    This is a consequence of \cite[Lemma 10.6 and Proposition 10.7]{Markman}, since the isometry~$\widetilde{H}(\mathfrak{F})$ is induced up to a sign by a self-isometry 
    \[
    \widetilde{H}(M_{S',v'}, \BQ) \simeq \widetilde{H}(S', \BQ) \to \widetilde{H}(S', \BQ)
    \]
    of the $K3$ surface $S'$ through the BKR correspondence, where the first isomorphism is given by the chosen isomorphism (\ref{iso}).
    \end{proof}

\subsection{Relations to other work}
\begin{enumerate}
    \item[(a)] Section \ref{sec1} focuses on generalizing the Fourier transform package from abelian varieties to hyper-K\"ahler varieties, where the role of the Poincaré line bundle is played by higher-rank hyperholomorphic bundles. In another direction, the Fourier transform package has been considered for \emph{abelian fibrations} with singular fibers where the Poincar\'e line bundle is replaced by certain generically rank $1$ Cohen--Macaulay sheaves; see \cite{MSY,MSY2} for more details. 
    \item[(b)] Markman's hyperholomorphic bundles, which are crucial in our construction of Section~\ref{sec1}, have also been applied to various other algebro-geometric questions for hyper-K\"ahler varieties of $K3^{[n]}$-type, including algebraic realizations of Hodge isometries \cite{Markman}, symmetries of derived categories \cite{KK, MSYZ, Zhang}, and the period--index problem \cite{HMSYZ, BH}.
    \item[(c)] A different Fourier transform was studied in \cite{SV} for the self-product of a hyper-K\"ahler fourfold which is either the Hilbert square of a $K3$ surface or the Fano variety of lines of a nonsingular cubic fourfold.
\end{enumerate}

\section{Proofs of Theorems \ref{thm0.4}, \ref{thm0.5}, and \ref{thm0.8}}\label{New_Sec2}

In this section, we complete the proofs of all the theorems concerning homological motives for hyper-K\"ahler varieties of $K3^{[n]}$-type: Theorems \ref{thm0.4}, \ref{thm0.5}, and \ref{thm0.8}. A key step is Proposition~\ref{prop:ffauto} which shows the existence of algebraic correspondences conjugating the cup-product and the convolution product for a hyper-K\"ahler variety of $K3^{[n]}$-type.

\subsection{Derived equivalences and Hodge isometries}

Let $X, X'$ be hyper-K\"ahler varieties of~$K3^{[n]}$-type. We say that $X, X'$ are Hodge isometric, if there is a Hodge isometry
\[
H^2(X, \BQ) \simeq H^2(X', \BQ)
\]
with respect to the BBF form. The following proposition shows that derived equivalent varieties of $K3^{[n]}$-type must be Hodge isometric.

\begin{prop}\label{D=>HI}
    If $X, X'$ are derived equivalent, then they are Hodge isometric.
\end{prop}

\begin{proof}
    The second rational cohomology of $X, X'$ can be decomposed into algebraic and transcendental parts
    \[
    H^2(X, \BQ) = T(X)_\BQ \oplus \mathrm{NS}(X)_\BQ, \quad H^2(X', \BQ) = T(X')_\BQ \oplus \mathrm{NS}(X')_\BQ.    \]
    Recall that the extended Mukai lattice for a hyper-K\"ahler variety $X$ admits a Hodge structure of weight 2 with
    \[
    \widetilde{H}^{2,0}(X)= H^{2,0}(X), \quad \widetilde{H}^{1,1}(X)= \BC \alpha \oplus H^{1,1}(X)\oplus \BC \beta, \quad \widetilde{H}^{0,2}(X) = H^{0,2}(X).
    \]
    Since $X, X'$ are derived equivalent, by \cite[Theorems 4.10 and 4.11]{Taelman} there is a Hodge isometry between the extended Mukai lattices
    \begin{equation*}
    \widetilde{H}(X, \BQ) \simeq \widetilde{H}(X', \BQ).
    \end{equation*}
    The Hodge isometry has to preserve the transcendental and the algebraic parts, which yields a Hodge isometry 
    \begin{equation}\label{tran}
    T(X)_\BQ \simeq T(X')_\BQ,
    \end{equation}
    and an isometry (as the Hodge structure is trivial on the algebraic part)
    \[
    \mathrm{NS}(X)_\BQ \oplus U \simeq \mathrm{NS}(X')_\BQ \oplus U, \quad U = \BQ\alpha \oplus \BQ \beta.
    \]
    By the Witt cancellation theorem, the lattice $\mathrm{NS}(X)_\BQ$ is isometric to $\mathrm{NS}(X')_\BQ$. Combined with the Hodge isometry (\ref{tran}), we obtain a Hodge isometry 
    \[
    H^2(X, \BQ) = T(X)_\BQ \oplus \mathrm{NS}(X)_\BQ \simeq T(X') \oplus \mathrm{NS}(X')_\BQ = H^2(X', \BQ). \qedhere
    \]
\end{proof}

The following corollary is immediate.

\begin{cor} \label{cor2.2}
    Theorem \ref{thm0.8} implies Theorem \ref{thm0.4}.
\end{cor}

\subsection{Fourier autoduality} \label{sec2.2}

Let $Y$ be the variety of $K3^{[n]}$-type introduced in Section \ref{sec1.2}. The goal is to relate the two graded algebra objects (both on the $Y$ side)
\begin{equation} \label{eq:hommot}
\left(\bigoplus_{k = 0}^{2n}h_{\mathrm{hom}, 2k}(Y), \cup\right), \quad \left(\bigoplus_{k = 0}^{2n}h_{\mathrm{hom}, 4n-2k}(Y)(2n-2k), \ast\right)
\end{equation}
in the category of homological motives.

\begin{rmk}
It is an open question whether any $K3^{[n]}$-type variety $Y$ admits a derived \emph{auto\-equivalence} with Fourier--Mukai kernel $\CF$ of nonzero rank. If~so, the algebraic correspondence~$\kappa(\CF)\sqrt{\mathrm{td}_{Y\times Y}}$ induces a degree-reversing automorphism
\[
\kappa(\CF)\sqrt{\mathrm{td}_{Y\times Y}}: H^*(Y, \BQ) \to H^*(Y, \BQ), \quad H^i(Y, \BQ) \to H^{4n - i}(Y, \BQ);
\]
see \cite[Lemma~4.1(1)]{Markman}. Also unknown is whether such a correspondence necessarily conjugates the cup-product and the convolution product on $H^*(Y, \BQ)$ after normalization; see~\cite[Example~1.7 and Conjecture~1.8]{Markman}.
\end{rmk}

Nevertheless, as we observe in this section, it is always possible to construct a degree-reversing, product-conjugating self-correspondence in $H^*(Y \times Y, \BQ)$, which is further induced by algebraic cycles. The construction relies on the Lefschetz standard conjecture for varieties of~$K3^{[n]}$-type established in \cite{CM, Markman}.

Choose an ample class $L \in H^2(Y, \BQ)$ of BBF norm $(L, L) = 2d$, and consider the associated~$\mathfrak{sl}_2$-triple $(e_L, f_L, h)$ with
\[
e_L = {}\cup L: H^*(Y, \BQ) \to H^{* + 2}(Y, \BQ), \quad h|_{H^i(Y, \BQ)} = (i - 2n)\mathrm{id}.
\]
Viewed as a correspondence, the Lefschetz operator $e_L$ is given by the cycle class
\begin{equation} \label{eq:defL}
e_L = \Delta_*L \in H^{4n + 2}(Y \times Y, \BQ)
\end{equation}
where $\Delta: Y \to Y \times Y$ is the diagonal morphism. By \cite[Lemma 1.6]{Markman}, the dual Lefschetz operator $f_L$ is also algebraic, and is explicitly given by the correspondence
\begin{equation} \label{eq:defLambda}
f_L = \mathfrak{F}^{-1} \circ e_{\pm\frac{1}{rd}\widetilde{H}(\mathfrak{F})^{-1}(L)} \circ \mathfrak{F} \in H^{4n - 2}(Y \times Y, \BQ).
\end{equation}
Here we follow the notation of Section \ref{sec1}: $\mathfrak{F}$ is the Fourier transform induced by the rank $n!r^n$ twisted vector bundle on $Y \times X$, and $\widetilde{H}(\mathfrak{F})$ is as in \eqref{eq:htilde} which for $n$ even is defined up to a~sign.

Now consider the correspondence
\[
\mathfrak{F}_L := \exp(e_L) \circ \exp(-f_L) \circ \exp(e_L) \in H^*(Y \times Y, \BQ)
\]
which we call the cohomological \emph{Fourier self-transform} associated with $L$. Note that if we view the $\mathfrak{sl}_2$-triple $(e_L, f_L, h)$ as elements of the LLV algebra $\mathfrak{so}(\widetilde{H}(X, \BC))$, then $\mathfrak{F}_L$ naturally lives in the associated Lie group $\mathrm{SO}(\widetilde{H}(Y, \BC))$ via the exponential map
\[
\mathrm{exp}: \mathfrak{so}(\widetilde{H}(X, \BC)) \to \mathrm{SO}(\widetilde{H}(Y, \BC)).
\]
The latter acts on $H^*(Y, \BC)$ by integrating the~LLV algebra action
\[
\mathrm{SO}(\widetilde{H}(Y, \BC)) \to \mathrm{GL}( H^*(Y, \BC) )\]
as in (\ref{integ}). In other words, the class $\mathfrak{F}_L \in H^*(Y\times Y, \BQ)$, viewed as an element in $\mathrm{GL}( {H}^*(Y, \BC))$, lies in the image of the integration map above. The inverse of $\FF_L$ is given by
\[
\FF_L^{-1} = \exp(-e_L) \circ \exp(f_L) \circ \exp(-e_L) \in H^*(Y \times Y, \BQ).
\]

We have by construction $\mathfrak{F}_L^{-1}h\mathfrak{F}_L = -h \in \mathfrak{so}(\widetilde{H}(X, \BC))$, and therefore
\[
\FF_L^{-1} \circ h \circ \FF_L = -h \in H^*(Y \times Y, \BQ).
\]
In particular, $\FF_L$ induces a degree-reversing automorphism
\[
\FF_L: H^*(Y, \BQ) \to H^*(Y, \BQ), \quad H^i(Y, \BQ) \to H^{4n - i}(Y, \BQ).
\]
Then as in Proposition \ref{prop1.8}(a), we obtain immediately the Fourier vanishing: that for any integers~$0 \leq i, j < 2n$ we have
\[
    \FF_{L, 2i + 1} = 0 \in H^{4i + 2}(Y \times Y, \BQ), \quad \FF^{-1}_{L, 2j + 1} = 0 \in H^{4j + 2}(Y \times Y, \BQ),
\]
and for any integers $0 \leq i, j \leq 2n$, we have
\[
    \mathfrak{F}_{L, 2i} \circ \mathfrak{F}^{-1}_{L, 2j}= 0 \in H^{4i+4j-4n}(Y \times Y, \BQ), \quad  \mathfrak{F}^{-1}_{L, 2i} \circ \mathfrak{F}_{L, 2j}= 0 \in H^{4i+4j-4n}(Y \times Y, \BQ).
\]
Here $\FF_{L, i}$ and $\FF^{-1}_{L, j}$ are the summands of $\FF_L, \FF_L^{-1}$ in $H^{2i}(Y \times Y, \BQ)$ and $H^{2j}(Y \times Y, \BQ)$ respectively.

The normalizations of the Fourier self-transforms $\FF_L, \FF_L^{-1}$ take the form
\begin{equation} \label{eq:normal2}
\begin{gathered}
\widetilde{\mathfrak{F}}_L = \sum_{k} \widetilde{\mathfrak{F}}_{L, 2k}, \quad \widetilde{\mathfrak{F}}_{L, 2k}: = d^{n - k}\mathfrak{F}_{L, 2k}, \\
\widetilde{\mathfrak{F}}^{-1}_L = \sum_{k} \widetilde{\mathfrak{F}}^{-1}_{L, 2k}, \quad \widetilde{\mathfrak{F}}^{-1}_{L, 2k}:= d^{n - k}\mathfrak{F}_{L, 2k},
\end{gathered}
\end{equation}
where $2d$ is the BBF norm of $L$.

Recall the two graded algebra objects from \eqref{eq:hommot}. We remark that although the two motivic decompositions of $h_{\mathrm{hom}}(Y)$ in \eqref{eq:hommot} (the second one with appropriate Tate twists) are obtained via different projectors, they are identical decompositions of homological motives. This is simply because modulo homological equivalence, the K\"unneth projectors are unique if they~exist.

The following result is key to the proofs of the main theorems.

\begin{prop} \label{prop:ffauto}
Let $Y$ be the variety of $K3^{[n]}$-type introduced in Section \ref{sec1.2}. Then the two graded algebra objects in \eqref{eq:hommot} are isomorphic through the Fourier self-transform $\widetilde{\FF}_L$.
\end{prop}

\begin{proof}
By the uniqueness of the K\"unneth projectors, the motivic decompositions of $h_{\mathrm{hom}}(Y)$ in~\eqref{eq:hommot} coincide with the one induced by the orthogonal projectors
\[
\FF_{L, 2k} \circ \FF^{-1}_{L, 4n - 2k} \in H^{4n}(Y \times Y, \BQ), \quad 0 \leq k \leq 2n,
\]
or the one induced by the orthogonal projectors
\[
\FF^{-1}_{L, 2k} \circ \FF_{L, 4n - 2k} \in H^{4n}(Y \times Y, \BQ), \quad 0 \leq k \leq 2n.
\]
Then, repeating the proof of Proposition \ref{prop1.3}, we obtain mutually inverse isomorphisms of homological motives
\[
h_{\mathrm{hom}, 4n - 2k}(Y)(2n - 2k) \xrightleftharpoons[~~~\mathfrak{F}^{-1}_{L, 4n-2k}~~~]{~~~\mathfrak{F}_{L, 2k}~~~} h_{\mathrm{hom}, 2k}(Y).
\]

Concerning the two product structures, we use \cite[Lemma 10.4]{Markman} which says that any element~$g \in O(\widetilde{H}(Y, \BC))$ satisfying 
\[
g^{-1}  h g = -h \in \mathfrak{so}(\widetilde{H}(Y,\BC))
\]
induces, after normalization, an automorphism of $H^*(Y, \BC)$ which conjugates the cup-product and the convolution product. In particular, this applies to $g = \FF_L$ the Fourier self-transform and its normalization~$\widetilde{\FF}_L$ in \eqref{eq:normal2}. Repeating again the proof of Proposition \ref{prop1.6}, we find commutative diagrams
\begin{equation*}
\begin{tikzcd}
h_{\mathrm{hom}, 4n - 2i}(Y)(2n - 2i) \otimes h_{\mathrm{hom}, 4n - 2j}(Y)(2n - 2j) \arrow[r, "\ast"] \arrow[d, "\widetilde{\mathfrak{F}}_{L, 2i} \otimes \widetilde{\mathfrak{F}}_{L, 2j}"] & h_{\mathrm{hom}, 4n - 2i - 2j}(Y)(2n - 2i - 2j) \arrow[d, "\widetilde{\mathfrak{F}}_{L, 2i + 2j}"] \\
h_{\mathrm{hom}, 2i}(Y) \otimes h_{\mathrm{hom}, 2j}(Y) \arrow[r, "\cup"] & h_{\mathrm{hom}, 2i+2j}(Y)
\end{tikzcd}
\end{equation*}
with the vertical arrows isomorphisms of homological motives. The direct sum of the diagrams above yields the desired isomorphism of graded algebra objects. 
\end{proof}

\begin{cor} \label{cor:isomxy}
Let $X, Y$ be varieties of $K3^{[n]}$-type which admit a (twisted) derived equivalence induced by a projectively hyperholomorphic bundle as in \eqref{FM}. Then there is an isomorphism of homological motives which respects the cup-product,
\[
(h_{\mathrm{hom}}(X), \cup) \simeq (h_{\mathrm{hom}}(Y), \cup).
\]
\end{cor}

\begin{proof}
The isomorphism is obtained as the composition of the isomorphism
\begin{equation} \label{eq:isom1}
\left(\bigoplus_{k = 0}^{2n}h_{\mathrm{hom}, 2k}(X), \cup\right) \simeq \left(\bigoplus_{k = 0}^{2n}h_{\mathrm{hom}, 4n-2k}(Y)(2n-2k), \ast\right)
\end{equation}
shown in Propositions \ref{prop1.6} and \ref{prop1.8}, and the isomorphism
\begin{equation} \label{eq:isom2}
\left(\bigoplus_{k = 0}^{2n}h_{\mathrm{hom}, 2k}(Y), \cup\right) \simeq \left(\bigoplus_{k = 0}^{2n}h_{\mathrm{hom}, 4n-2k}(Y)(2n-2k), \ast\right)
\end{equation}
shown in Proposition \ref{prop:ffauto}.
\end{proof}

\begin{rmk} \label{rmkfrob}
    We explain here that the isomorphism we construct for Corollary \ref{cor:isomxy} sends the class of a point $[\mathrm{pt}] \in H^{4n}(X, \BQ)$ to $[\mathrm{pt}] \in H^{4n}(Y, \BQ)$. This is deduced from the following~facts.
\begin{enumerate}
\item[(a)] The unit of the convolution product on $H^*(Y, \BQ)$ is $[\mathrm{pt}]/n!$; this is \cite[Lemma 10.3]{Markman}.

\item[(b)] The Fourier transform $\widetilde{\FF}$ in cohomology, per our convention, sends the convolution product on $H^*(Y, \BQ)$ to the cup-product on $H^*(X, \BQ)$. At the same time it sends the cup-product on $H^*(Y, \BQ)$ to the convolution product on $H^*(X, \BQ)$, the latter being the convention adopted by Markman. This is because we have
\[
\FF^{-1} = {^t\FF} \in H^*(X \times Y, \BQ)
\]
by \cite[Corollary 4.2]{Markman}, which uses the vanishing of the odd cohomology.

\item[(c)] The Fourier self-transform $\widetilde{\FF}_L$ also interchanges the cup-product and the convolution product on $H^*(Y, \BQ)$. Again this is because
\[
\FF_L^{-1} = {\FF}_L \in H^*(Y \times Y, \BQ).
\]
In fact, we have by definition $(\FF_L)^2 = -\mathrm{id} \in \mathrm{SL_2}$, which gets sent to $\mathrm{id} \in \mathrm{SO}(\widetilde{H}(Y, \BC))$ since $\widetilde{H}(Y, \BC)$ has even weights.
\end{enumerate}
Therefore, the composition of the isomorphism \eqref{eq:isom1} and \eqref{eq:isom2} sends $[\mathrm{pt}] \in H^{4n}(X, \BQ)$ (which is $n!$ times the unit of the convolution product on $H^*(X, \BQ)$) first to $n![Y] \in H^0(Y, \BQ)$ and then to $[\mathrm{pt}] \in H^{4n}(Y, \BQ)$.
\end{rmk}

\subsection{Proofs of Theorems \ref{thm0.5} and \ref{thm0.8}} \label{sec2.3}

In view of Corollary \ref{cor2.2}, it remains to prove Theorems \ref{thm0.5} and \ref{thm0.8}. The proof of Theorem \ref{thm0.8} relies on the following characterization, due to Markman \cite[Theorem 1.4]{Markman}, of Hodge isometric varieties of $K3^{[n]}$-type.

\begin{thm}[Markman \cite{Markman}] \label{thm:classify}
Let $X, X'$ be Hodge isometric varieties of $K3^{[n]}$-type. Then there exists a finite sequence of varieties of $K3^{[n]}$-type $X = X_1, X_2, \ldots, X_t = X'$ such that each pair $X_i, X_{i + 1}$ satisfies either
\begin{enumerate}
\item[(a)] there is a derived equivalence $D^b(X_i, \alpha_{X_i}) \simeq D^b(X_{i + 1}, \alpha_{X_{i + 1}})$ induced by a projectively hyperholomorphic bundle as in \eqref{FM}, or
\item[(b)] there is a parallel transport $H^2(X_i, \BZ) \to H^2(X_{i + 1}, \BZ)$ preserving the Hodge structures.
\end{enumerate}
\end{thm}

\begin{proof}[Proof of Theorem \ref{thm0.8}]
By Theorem \ref{thm:classify}, it suffices to construct an isomorphism of homological motives respecting the cup-product for each pair $X_i, X_{i + 1}$ in the sequence. In case (a) the isomorphism is given by Corollary \ref{cor:isomxy}. In case (b) the two varieties are necessarily birational by the global Torelli theorem \cite[Theorem 1.3]{Torelli}. We then conclude by \cite{Riess}.
\end{proof}

We move on to Theorem \ref{thm0.5}. Let $S$ be a projective $K3$ surface, and let $v \in \widetilde{H}(S, \BZ)$ be a primitive Mukai vector with $v^2 = 2n-2\geq 0$. The moduli space $M_{S,v}$ parameterizes $H$-stable sheaves on $S$ with Mukai vector $v$, where the polarization $H$ is $v$-generic. We now prove that there is an isomorphism of homological motives which respects the cup-product,
\begin{equation*}\label{main_thm_sec2}
(h_{\mathrm{hom}}(S^{[n]}), \cup) \simeq (h_{\mathrm{hom}}(M_{S,v}), \cup).
\end{equation*}
This clearly implies Theorem \ref{thm0.5}.

Again by \cite{Riess}, it suffices to consider the case when $M_{S,v}$ is not birational to~$S^{[n]}$. In this case, Zhang proved in \cite{Zhang} that there is a derived equivalence
\begin{equation}\label{Zhang_FM}
D^b(S^{[n]}) \simeq D^b(M_{S,v}, \theta_{v}),
\end{equation}
where $\theta_v$ is a particular Brauer class on $M_{S,v}$. More precisely, Zhang's equivalence (\ref{Zhang_FM}) can be realized as the composition~of 
\begin{equation}\label{Zhang_FM2}
\Phi_{(\CE,\theta_v)}: D^b(S^{[n]}) \xrightarrow{\simeq} D^b(M'_{S,v}, \theta_v)
\end{equation}
and
\begin{equation}\label{D-equiv}
D^b(M'_{S,v}, \theta_{v}) \xrightarrow{\simeq} D^b(M_{S,v}, \theta_v).
\end{equation}
Here $M'_{S,v}$ is a hyper-K\"ahler variety birational to $M_{S,v}$, and we use the same notation $\theta_v$ to denote the Brauer class on $M'_{S,v}$ obtained via the birational map. The equivalence (\ref{Zhang_FM2}) is a special case of~(\ref{FM}) induced by a projectively hyperholomorphic bundle $\CE$ on $S^{[n]} \times M'_{S,v}$, and~(\ref{D-equiv}) is given by the proven (twisted) $D$-equivalence conjecture \cite{DHL, MSYZ} for $K3^{[n]}$-type.

\begin{proof}[Proof of Theorem \ref{thm0.5}]
Again it suffices to construct an isomorphism of homological motives respecting the cup-product for the pair $S^{[n]}, M'_{S, v}$, and for the pair $M'_{S, v}, M_{S,v}$ respectively. The former is given by Corollary \ref{cor:isomxy}, while the latter follows from \cite{Riess}.
\end{proof}

\begin{rmk}
    Another approach is to reduce Theorem \ref{thm0.5} to Theorem \ref{thm0.8}: this is due to the fact that $M_{S,v}$ and $M_{S,v'}$ are Hodge isometric as long as $v^2=v'^2$. The proof we gave above relies on the recent work of Zhang \cite{Zhang} and provides an explicit isomorphism connecting $S^{[n]}$ and $M_{S,v}$; this isomorphism will be used in the proof of Theorem \ref{thm0.7}. 
\end{rmk}

\section{Generically defined kappa classes}\label{sec2}

Throughout this section, we assume that $(S,H)$ is a \emph{fixed} $K3$ surface of genus $g \geq 2$ with~$\mathrm{Pic}(S)=\BZ H$. Let 
\[
v = (r, mH, s) \in \widetilde{H}(S, \BZ), \quad r, m, s \in \BZ
\]
be a primitive Mukai vector with $v^2 = 2n-2\geq 0$. The main purpose of this section is to construct generically defined kappa class; see Theorem \ref{thm_main2}.

\subsection{Moduli of polarized $K3$ surfaces}\label{sec2.1}

Recall the derived equivalences (\ref{Zhang_FM2}) and (\ref{D-equiv}) for the fixed polarized $K3$ surface $(S,H)$. Now we vary $(S,H)$ over the moduli stack $\CX_g$. We~denote
\[
\CX^*_g \subset \CX_g
\]
the open substack formed by polarized $K3$ surfaces whose polarization is general with respect to the Mukai vector $v$. Clearly all polarized $K3$ surfaces of Picard rank $1$ lie in $\CX^*_g$. We assume that the original $K3$ surface $(S,H)$ is represented by a point $P \in \CX^*_g$. The universal family~$\CS_g \to \CX^*_g$ carries a relative polarization $\CH \in \mathrm{Pic}(\CS_g)$ of genus $g$, and we have the proper and smooth families
\begin{equation}\label{families}
\CS_g^{[n]} \to \CX^*_g, \quad \CM_{g,v} \to \CX^*_g
\end{equation}
formed by the Hilbert schemes of points and the moduli of stable sheaves with respect to the Mukai vector $v$ respectively. Here the Mukai vector and the polarization are defined relatively over $\CX_g$ by the relative polarization $\CH$; the general assumption of the polarization ensures that~$\CM_{g,v} \to \CX^\ast_g$ is smooth with hyper-K\"ahler fibers. The fibers of (\ref{families}) over $P \in \CX^*_g$ recover the hyper-K\"ahler varieties $S^{[n]}$ and $M_{S,v}$ respectively.

We may modify the family $\CM_{g,v} \to \CX^*_g$ to include the birational model $M'_{S,v}$ used in (\ref{Zhang_FM2}) and (\ref{D-equiv}) as a fiber.

\begin{prop}\label{prop2.1}
After shrinking $\CX^*_g$ to an open Deligne--Mumford substack $\CX^\circ_g \subset \CX^*_g$ containing $P$, there is a proper and smooth morphism 
\[
\Pi_M: \CM'_{g,v} \to \CX^\circ_g
\]
satisfying
\begin{enumerate}
    \item[(a)] the fiber over $P\in \CX^\circ$ recovers the birational model $M'_{S,v}$ of $M_{S,v}$, and 
    \item[(b)] each fiber of $\Pi_M$ is birational to the corresponding fiber of the family \mbox{$\CM_{g,v} \to \CX^\circ_g$} given in \eqref{families}.
\end{enumerate}
\end{prop}

\begin{proof}
By \cite{BM}, the birational model $M'_{S,v}$ of $M_{S,v}$ is realized as the moduli space of stable objects on $S$ with Mukai vector $v$ with respect to a $v$-generic Bridgeland stability condition. The proposition follows from the general construction of families of Bridgeland stability conditions~\cite{Six}; see the proof of \cite[Theorem 29.4(2)]{Six}.
\end{proof}

We consider the following family given by (\ref{families}) and Proposition \ref{prop2.1},
\begin{equation}\label{family2}
\Pi: \CS_g^{[n]}\times_{\CX^\circ} \CM'_{g,v} \to \CX^\circ_g,
\end{equation}
whose fiber over $P \in \CX^\circ_g$ is the product $S^{[n]} \times M'_{S,v}$. We say that a class
\[
\gamma \in H^*(S^{[n]} \times M'_{S,v}, \BQ)
\]
is \emph{monodromy invariant}, if it is invariant under the geometric monodromy action for the family~(\ref{family2}).

\begin{lem}\label{lem2.2}
    All the classes in 
    \[
    \mathrm{Pic}(S^{[n]} \times M'_{S,v})_\BQ = \mathrm{Pic}(S^{[n]})_\BQ \oplus \mathrm{Pic}(M'_{S,v})_\BQ \subset H^2(S^{[n]} \times M'_{S,v}, \BQ)
    \]
    are monodromy invariant.
\end{lem}

\begin{proof}
There are two algebraic classes in $H^2(M_{S,v}, \BQ) = H^2(M'_{S,v}, \BQ)$ which are given by the two algebraic classes in the orthogonal complement 
\[
v^\perp \subset \widetilde{H}(S, \BQ)
\]
via a correspondence given by the universal family. Therefore, both classes can be defined relatively over $\CX^\circ_g$; this proves the monodromy invariance of the classes in $\mathrm{Pic}(M'_{S,v})_\BQ$. The monodromy invariance for $S^{[n]}$ is parallel and simpler.
\end{proof}

The Fourier--Mukai kernel of (\ref{Zhang_FM2}) induces the (cohomological) kappa class
\begin{equation}\label{kappa}
\kappa(\CE) \in H^*(S^{[n]} \times M'_{S,v}, \BQ)
\end{equation}
as in Section \ref{sec1.2} which, by the K\"unneth formula, is completely determined by its induced (cohomological) Fourier transform
\begin{equation}\label{Fourier2}
    \mathfrak{F}:H^*(S^{[n]}, \BQ) \to H^*(M'_{S,v}, \BQ).
\end{equation}

\begin{prop}\label{prop2.3}
    The class \eqref{kappa} is monodromy invariant. In particular, \eqref{kappa} defines a class 
    \begin{equation}\label{coh}
\kappa(\CE) \in H^0\left(\CX^\circ_g, R\pi_* \BQ_{\CS_g^{[n]}\times_{\CX^\circ} \CM'_{g,v}}\right)
\end{equation}
via parallel transport.
\end{prop}

\begin{proof}
    By Lemma \ref{lem1.10}, the morphism (\ref{Fourier2}) is determined by its induced morphism on the extended Mukai lattices
    \begin{equation}\label{extended}
    \widetilde{H}(\mathfrak{F}):  \widetilde{H}(S^{[n]}, \BQ) \to  \widetilde{H}(M'_{S,v}, \BQ).
    \end{equation}
    Hence, it suffices to prove that (\ref{extended}) is monodromy invariant for (\ref{family2}). 

    The morphism (\ref{extended}) preserves the algebraic and the transcendental parts respectively. The morphism on the transcendental part is determined by the class of the symplectic form, which is clearly monodromy invariant. The algebraic parts of $\widetilde{H}(S^{[n]}, \BQ)$ and $\widetilde{H}(M'_{S,v}, \BQ)$ are spanned by $\alpha, \beta$ and the classes in $\mathrm{Pic}(S^{[n]})_\BQ$ and $\mathrm{Pic}(M'_{S,v})_\BQ$ respectively. By Lemma \ref{lem2.2}, all these classes are monodromy invariant. Therefore the morphism on the algebraic part is also monodromy invariant. This completes the proof.
\end{proof}

The main result of Section \ref{sec2} is the following theorem, whose proof will be completed in Section \ref{sec_proof}.

\begin{thm}\label{thm_main2}
    There is a class
    \[
    \mathfrak{K}^{\mathrm{univ}} \in \mathrm{CH}^*\left(\CS_g^{[n]}\times_{\CX_g^\circ} \CM'_{g,v}\right)
    \]
    whose restriction to every fiber over $\CX^\circ_g$ recovers the cohomology class \eqref{coh}.
\end{thm}

The Fourier--Mukai transform (\ref{Zhang_FM2}) is constructed for any $K3$ surface using transcendental techniques (\emph{e.g.}~twistor lines, hyper-holomorphic bundles, \emph{etc}.). In order to apply the Franchetta properties over the moduli of polarized $K3$ surfaces, we want to globalize it over a Zariski open subset of $\CX_g$. This is achieved by a moduli approach which we discuss in the~following.

\subsection{Stable vector bundles on gerbes}
The Brauer class $\theta_v$ in Zhang's Fourier--Mukai transform (\ref{Zhang_FM2}) is given by a class
\begin{equation}\label{theta}
\theta_v \in H^2(M'_{S,v}, \BZ/(2n-2)\BZ).
\end{equation}
Geometrically, this corresponds to a $\mathbf{\mu}_{2n-2}$-gerbe
\[
\mathfrak{M}'_{S,v} \to M'_{S,v}.
\]

The class (\ref{theta}) is canonical up to a sign (see \cite[Section 2.2]{Zhang}). Therefore it can be defined as a relative class for the family $\Pi_M: \CM'_{g,v} \to \CX^\circ_g$; in other words, it gives rise to a section 
\begin{equation}\label{theta2}
[\theta_v] \in H^0( \CX^\circ_g, R^2{\Pi_M}_*(\BZ/(2n-2)\BZ)),
\end{equation}
which yields a $\mu_{2n-2}$-gerbe for each fiber of $\Pi_M: \CM'_{g,v} \to \CX^\circ_g$.

The next lemma shows that we can globalize these gerbes by passing to a generically finite cover.

\begin{lem}\label{lem2.4}
        There exists an irreducible nonsingular quasi-projective variety ${\CX}^\natural_g$ together with a generically finite morphism
        \[
        \nu: {\CX}^\natural_g \to \CX^\circ_g
        \]
        such that the pullback
        \[
        [\theta_v] \in H^0( \CX^\natural_g, R^2{\Pi_M}_*\left(\BZ/(2n-2)\BZ\right)).        \]
        of \eqref{theta2} along $\nu$ can be lifted to a class
        \[
        \Theta_v \in H^2(\CM'^{\natural}_{g,v}, \BZ/(2n-2)\BZ).
        \]
        Here we use $\Pi_M: \CM'^{\natural}_{g,v} \to {\CX}^\natural_g$ to denote the pullback of the family $\Pi_M: \CM'_{g,v} \to \CX^\circ_g$ via $\nu$. 
    \end{lem}

\begin{proof}
    Using level structures, we can dominate the irreducible nonsingular Deligne--Mumford stack $\CX^\circ_{g} \subset \CX_g$ by an irreducible quasi-projective nonsingular variety $\CX^\natural_g$ through a generically finite morphism. However, it is not clear whether the section $[\theta_v]$ can be lifted to a global class $\Theta_v$ as desired; this is because the Leray spectral sequence with $\BZ/(2n-2)\BZ$-coefficients 
    \[
H^i( \CX^\natural_g, R^j{\Pi_M}_*(\BZ/(2n-2)\BZ))  \Rightarrow H^{i+j}  (\CM'^{\natural}_{g,v}, \BZ/(2n-2)\BZ)
\]
    may not degenerate on the second page. 

   The obstruction to lifting $[\theta_v]$ is given by the image of the boundary map
    \[
    d_3([\theta_v]) \in H^3(\CX^\natural_g, \BZ/(2n-2)\BZ).
    \]
    This class may be annihilated by further pulling back along a finite morphism. Replacing $\CX^\natural_g$ by the domain of this finite morphism completes the proof.
\end{proof}

By Lemma \ref{lem2.4}, we obtain a $\mu_{2n-2}$-gerbe
\begin{equation}\label{global_gerbe}
\mathfrak{M}'^\natural_{g,v} \to \CM'^{\natural}_{g,v} \to \CX_g^\natural
\end{equation}
whose restriction to every fiber recovers the $\mu_{2n-2}$-gerbe induced by $[\theta_v]$. We also consider the family obtained from the pullback of the Hilbert schemes,
\[
\CS^{[n]}_g \to \CX^\natural_g,
\]
and the relative product
\begin{equation}\label{family3}
\CS^{[n]}_g \times_{\CX^\natural_g} \mathfrak{M}'^\natural_{g,v} \to {\CX^\natural_g}.
\end{equation}

Assume that $Q \in \CX_g^\natural$ is a point over $P \in \CX^\circ_g$, so that the fiber over $Q$ of the family (\ref{family3}) recovers~$S^{[n]} \times \mathfrak{M}'_{S,v}$. The families (\ref{global_gerbe}) and (\ref{family3}) play a key role in the proof of Theorem \ref{thm_main2}.

\subsection{Proof of Theorem \ref{thm_main2}}\label{sec_proof}

The Fourier--Mukai kernel $(\CE, \theta_v)$ is a twisted vector bundle on~$S^{[n]} \times M'_{S,v}$. Since the twisted vector bundle $(\CE, \theta_v)$ is deformed along a twistor path from a slope stable vector bundle on a pair of Hilbert schemes, it is slope stable with respect to an ample class
\[
A\in \mathrm{Pic}(S^{[n]} \times M'_{S,v})_\BQ.
\]
Equivalently, the twisted vector bundle $(\CE, \theta_v)$ on $S^{[n]} \times M'_{S,v}$ corresponds to an $A$-slope stable vector bundle on the $\mu_{2n-2}$-gerbe:
\begin{equation}\label{vector_bundle}
\CE \in  \mathrm{Coh}(S^{[n]} \times \mathfrak{M}'_{S,v}).
\end{equation}
By Lemma \ref{lem2.2}, the polarization $A$ on $S^{[n]} \times M'_{S,v}$ can be lifted to a relative polarization
\[
\CA \in \mathrm{Pic}\left(\CS^{[n]}_g \times_{\CX^\natural_g} \CM'^\natural_{g,v}\right)_\BQ.
\]
Let $\mathfrak{T}^\natural_g \to \CX^\natural_g$ be the relative moduli stack of vector bundles on the fibers of (\ref{family3}) satisfying 
\begin{enumerate}
    \item[(a)]  they are slope stable with respect to the polarization $\CA$, and
    \item[(b)] the kappa classes of the vector bundles are given by (\ref{coh}).
\end{enumerate}
We rigidify the moduli stack with respect to the scaling automorphism $\mathbb{C}^*$, so that it is Deligne--Mumford. For convenience, from now on we use $\mathfrak{T}^\natural_g$ to denote this Deligne--Mumford stack. Clearly the vector bundle (\ref{vector_bundle}) lies in the moduli stack $\mathfrak{T}^\natural_g$ over the point $Q \in \CX^\natural_g$:
\[
[\CE] \in \mathfrak{T}^\natural_g \mapsto Q \in \CX^\natural_g.
\]
In particular, the moduli stack $\mathfrak{T}^\natural_g$ is non-empty.

\begin{lem}\label{lem2.6}
    The moduli stack $\mathfrak{T}^\natural_g$ is a $\mu_{2n-2}$-gerbe over an algebraic space $\CT^\natural_g$ over $\CX^\natural_g$, 
    \[
    \mathfrak{T}^\natural_g \to  \CT^\natural_g \to  \CX^\natural_g,
    \]
    where $\CT^\natural_g$ is a countable union of finite type algebraic spaces over $\CX_g^\natural$.
\end{lem}

\begin{proof}
The (rigidified) moduli stack of stable vector bundles on the $\mu_{2n-2}$-gerbe $\mathfrak{T}^\natural_g$ with fixed Chern character is a $\mu_{2n-2}$-gerbe over a finite type algebraic space, \emph{c.f.}~\cite[Proposition~2.3.3.4]{Lie}. Here we only fix the kappa class, which determines all the Chern class of the vector bundle up to a class in the second rational cohomology, \emph{i.e.}, $\kappa(\CE) = \kappa(\CE')$ implies $\mathrm{ch}(\CE)= \mathrm{ch}(\CE)\cup e^{\ell}$ with $\ell$ a class in the second rational cohomology. This ambiguity allows the moduli space to be a countable union of algebraic spaces of finite type.
\end{proof}

\begin{prop}\label{prop2.7}
    The moduli stack $\mathfrak{T}^\natural_g$ dominates $\CX^\natural_g$.
\end{prop}

\begin{proof}
It suffices to show that $\mathfrak{T}^\natural_g$ dominate an analytic neighborhood of $Q \in \CX^\natural_g$. Equivalently, we consider the moduli point $P=[(S,H)] \in \CX_g$, and let $(S_t, H_t)$ be a polarized $K3$ surface representing a very general point in an analytic neighborhood $\mathbf{\Delta}$ of $P \in \CX_g$; it suffices to show that there is a $\theta_v$-twisted slope stable vector bundle on $S^{[n]}_t \times M'_{S_t,v}$ with respect to the polarization $\CA$ whose kappa class coincides with (\ref{coh}). Here the polarization $\CA$, the class $\theta_v$, and the Mukai vector $v$ are all defined relatively over $\CX^\circ_g$; therefore they are all defined for the nearby polarized $K3$ surface $(S_t, H_t)$.

Recall from \cite{Markman} that the Fourier--Mukai transform (\ref{Zhang_FM2}) induces a Hodge isometry
\[
\phi: H^2(S^{[n]}, \BQ) \to H^2(M'_{S,v}, \BQ)
\]
which sends a K\"ahler class on $S^{[n]}$ to a K\"ahler class on $M'_{S,v}$. Indeed, this condition concerning the K\"ahler classes is exactly the reason we have to consider the birational model $M'_{S,v}$ of~$M_{S,v}$. We consider the parallel transport along a path connecting the nearby point $[(S_t, H_t)]$ to~$[(S,H)]$ in $\mathbf{\Delta} \subset \CX_g$:
\begin{equation}\label{PT6}
\mathrm{PT}_{S^{[n]}}: H^2(S_t^{[n]}, \BQ) \to H^2(S^{[n]}, \BQ), \quad \mathrm{PT}_M: H^2(M'_{S_t,v}, \BQ) \to H^2(M'_{S,v}, \BQ).
\end{equation}
If the composition
\begin{equation}\label{Kahler}
\mathrm{PT}^{-1}_M\circ\varphi\circ \mathrm{PT}_{S^{[n]}}:  H^2(S_t^{[n]}, \BQ) \to H^2(M'_{S_t,v}, \BQ)
\end{equation}
also sends a K\"ahler class to a K\"ahler class, then Zhang's construction \cite{Zhang} shows that there is a $\theta_v$-twisted vector bundle $\CE_t$ on $S^{[n]}_t \times M'_{S_t, v}$ which is connected to $(\CE ,\theta_v)$ on $S^{[n]} \times M'_{S,v}$ by a twistor path in the sense of Markman  \cite{Markman}, and it induces a derived equivalence
\[
\Phi_{\CE_t, \theta_v}: D^b(S_t^{[n]}) \xrightarrow{\simeq} D^b(M'_{S,v}, \theta_v).
\]
Hence $\kappa(\CE_t) \in H^*(S_t^{[n]}\times M'_{S_t,v}, \BQ)$ is sent to $\kappa(\CE) \in H^*(S^{[n]} \times M'_{S,v}, \BQ)$ via a parallel transport; in particular the kappa class of $\CE_t$ coincides with the class (\ref{coh}). Therefore $\CE_t$ gives the desired bundle over $S^{[n]}_t\times \mathfrak{M}'_{S_t,v}$.

By the discussion above, to complete the proof, we need to show that there is an analytic neighborhood $P \in \mathbf{\Delta}$ whose general point $[(S_t, H_t)]$ satisfies that (\ref{Kahler}) sends a K\"ahler class to a K\"ahler class. This can be achieved since (\ref{PT6}) do not change the K\"ahler cones for a general point $[(S_t, H_t)]$ in a sufficiently small neighborhood of $P = [(S, H)]$.
\end{proof}

\begin{proof}[Proof of Theorem~\ref{thm_main2}]
Combining Lemma \ref{lem2.6} and Proposition \ref{prop2.7}, we can find an irreducible component of the moduli space $\CT^\natural_g$ dominating $\CX^\natural_g$. By taking sub-spaces and resolving the singularities, there is an irreducible nonsingular algebraic space $\CR^\natural_g$ over $\CT^\natural_g$ which satisfies that the composition
\[
\CR^\natural_g \to \CT_g^\natural \to \CX^\natural_g
\]
is generically finite. The pullback of the $\mu_{2n-2}$-gerbe $\mathfrak{T}^\natural_g \to \CT^\natural_g$ yields a $\mu_{2n-2}$-gerbe
\[
\mathfrak{R}^\natural_g \to \CR^\natural_g,
\]
and there is a universal vector bundle 
\[
\CE^\mathrm{univ} \in \mathrm{Coh}\left( \CS^{[n]}_g \times_{\CX^\natural_g} \mathfrak{M}'_{g,v} \times_{\CX^\natural_g} \mathfrak{R}^\natural_g \right)
\]
pulled back from the universal vector bundle over the moduli stack $\mathfrak{T}^\natural_g$.

We consider the composition
\[
\epsilon: \CS^{[n]}_g \times_{\CX^\natural_g} \mathfrak{M}'_{g,v} \times_{\CX^\natural_g} \mathfrak{R}^\natural_g \to  \CS^{[n]}_g \times_{\CX_g} {\CM}'_{g,v}
\]
of the natural morphisms
\[
 \CS^{[n]}_g \times_{\CX^\natural_g} \mathfrak{M}'_{g,v} \times_{\CX^\natural_g} \mathfrak{R}^\natural_g \to  \CS^{[n]}_g \times_{\CX^\natural_g} \mathfrak{M}'_{g,v}
\]
and 
\[
\CS^{[n]}_g \times_{\CX^\natural_g} \mathfrak{M}'_{g,v} \to \CS^{[n]}_g \times_{\CX^\natural_g} {\CM}'_{g,v} \to \CS^{[n]}_g \times_{\CX^\circ_g} {\CM}'_{g,v}.
\]
The desired class of Theorem \ref{thm_main2} is given by averaging the kappa class of the bundle $\CE^{\mathrm{univ}}$,
\[
\mathfrak{K}^{\mathrm{univ}} := \frac{1}{\mathrm{deg}(\epsilon)} \epsilon_*\kappa( \CE^{\mathrm{univ}}) \in \mathrm{CH}^*\left(\CS_g^{[n]}\times_{\CX_g^\circ} \CM'_{g,v}\right),
\]
which is well-defined since the degree of $\epsilon$ is finite. This completes the proof.
\end{proof}

\begin{rmk}
 Since there are several spaces involved in the proof, we summarize the reasons of introducing them for the reader's convenience. The original purpose concerns ``spreading out'' a cycle on $S^{[n]} \times M_{S,v}$ over a Zariski open subset of $\CX_g$. We first switch to a birational model~$M'_{S,v}$ so that the cycle is induced by a twisted vector bundle; vector bundles are easier to be deformed. We take the open subset $\CX^\circ_g \subset \CX_g$ to ``spread out'' the geometric fiber~$S^{[n]} \times M'_{S,v}$. Next, we switch to working over $\CX^\natural_g$ in order to lift the gerbe of each fiber to a global gerbe. Then we use Zhang's construction \cite{Zhang} to show that a certain relative moduli of stable bundles on the gerbe dominates the base $\CX^\natural_g$, which constructs the desired cycle class. Note that by our construction the restriction of $\mathfrak{K}^\mathrm{univ}$ over the moduli point $P=[(S,H)] \in \CX_g$ may not exactly recover $\kappa(\CE) \in \mathrm{CH}^*(S^{[n]} \times M'_{S,v})$. Nevertheless we know from the monodromy invariance (Proposition \ref{prop2.3}) that they are matched as cohomology classes, which is enough for our purpose. 
\end{rmk}

\section{Franchetta properties and the proof of Theorem \ref{thm0.7}} \label{New_Sec4}

We complete the proof of Theorem \ref{thm0.7} by applying the Fourier transform package of Section~\ref{sec1} to the universal kappa class given by Theorem \ref{thm_main2}.

\subsection{Franchetta properties}
Assume $g\geq 2$. We consider the universal family $\CS_g \to \CX_g$ over the moduli stack of $K3$ surfaces of genus $g$. Recall the \emph{generically defined cycles}
\begin{equation*}
\mathrm{GDCH}^*(S^n):= \mathrm{Im}\left( \mathrm{CH}^*({\CS^n_g}_{/\CX_g}) \to \mathrm{CH}^*(S^n) \right)
\end{equation*}
on the $n$-th product of any fiber $S\subset \CS_g$. In fact, by a standard spreading out argument, a class in $\mathrm{CH}^*(S^n)$ is generically defined if it is given by the restriction of a class on the $n$-th relative product over a Zariski dense open subset $\CU \subset \CX_g$. The $n$-th Franchetta property holds for $K3$ surfaces of genus $g$ if the cycle class map is injective restricting to the subgroup 
\[
\mathrm{GDCH}^*(S^n) \subset \mathrm{CH}^*(S^n).
\]
The following conjecture, which generalizes O'Grady's generalized Franchetta conjecture \cite{OG}, was studied in \cite{FLV, FLV2} and stated explicitly in \cite[Conjecture 1.5]{LV}. 

\begin{conj}[Generalized Franchetta conjecture]\label{conj3.1}
For any $g\geq 2$ and $n\geq 1$, the $n$-th Franchetta property holds for $K3$ surfaces of genus $g$.
\end{conj}

We can consider similarly generically defined cycles on (products of) nonsingular moduli spaces of stable sheaves on a $K3$ surface, which are defined relatively over a Zariski open subset of $\CX_g$. The following result given by \cite[Theorem 4.1]{FLV2} reduces the Franchetta properties for moduli spaces of stable sheaves to the Franchetta properties for products of $K3$ surfaces.

\begin{thm}[\cite{FLV2}]\label{thm3.2}
Assume that Conjecture \ref{conj3.1} holds for $g,n$. Let $(S, H)$ be a $K3$ surface of genus~$g$ with $\mathrm{Pic}(S) = \BZ H$. Then the cycle class map is injective when restricted to the~subgroup 
\[
\mathrm{GDCH}^*(M_1 \times M_2 \times \cdots \times M_t) \subset \mathrm{CH}^*(M_1 \times M_2 \times \cdots \times M_t)
\]
if $\dim M_1 + \dim M_2 + \cdots \dim M_t \leq 2n$. Here each $M_i$ is a hyper-K\"ahler variety birational to a moduli space of $H$-stable sheaves on $S$ with respect to a primitive Mukai vector $v_i$.
\end{thm}

\begin{rmk}
   The proof of \cite[Theorem 4.1]{FLV2} only treats the the case where $M_i$ is a moduli space of stable sheaves. However, by \cite{BM} any hyper-K\"ahler birational model of the moduli of stable sheaves is given by the moduli of stable objects with respect to a Bridgeland stability condition, and the proof works identically in this setting; see the proof of \cite[Theorem~3.1]{FLV2}. 
\end{rmk}

\subsection{Generically defined Fourier transforms}

We fix $(S,H)$ to be a $K3$ surface of genus~$g$ with $\mathrm{Pic}(S) = \BZ H$. Under the notation of Section \ref{sec2.1}, the pair $(S,H)$ is represented by a moduli point $P:= [(S,H)] \in \CX^\circ_g$.

Recall the hyper-K\"ahler varieties $M_{S,v}, M'_{S,v}, S^{[n]}$ with 
\[
\dim M_{S,v} = \dim M'_{S,v} = 2n,
\]
and the Fourier--Mukai transform (\ref{Zhang_FM2}). We define the generically defined kappa class
\begin{equation}\label{GD_kappa}
\kappa_v^{\mathrm{GD}} \in \mathrm{GDCH}^*(S^{[n]} \times M'_{S,v})
\end{equation}
as the restriction of the class $\mathfrak{K}^{\mathrm{univ}}$ in Theorem \ref{thm_main2} to the fiber over $P \in \CX^\circ_g$. We have
\begin{equation}\label{coh_match}
\kappa(\CE) = \kappa_v^{\mathrm{GD}} \in H^*(S^{[n]} \times M'_{S,v}, \BQ).
\end{equation}
Following Section \ref{sec1.2}, the class (\ref{GD_kappa}) induces the generically defined Fourier transforms
\begin{equation}\label{Fourier_GD}
\mathfrak{F}_v^{\mathrm{GD}} \in \mathrm{GDCH}^*(S^{[n]} \times M'_{S,v}), \quad \mathfrak{F}_v^{\mathrm{GD,-1}} \in  \mathrm{GDCH}^*(M'_{S,v} \times S^{[n]})
\end{equation}
where we replace $\kappa(\CE)$ by $\kappa_v^{\mathrm{GD}}$. Similarly, we define the normalizations $\widetilde{\mathfrak{F}}_v^{\mathrm{GD}}, {{\widetilde{\mathfrak{F}}_v}^{\mathrm{GD}, -1} }$ and the convolution class
\[
\mathfrak{C}^{\mathrm{GD}}_v \in \mathrm{GDCH}^*(S^{[n]} \times S^{[n]} \times S^{[n]} )
\]
following Section \ref{sec1.3}.

The following proposition shows that Conjectures \ref{conj1.1} and \ref{conj1.4} for the Fourier transform package $(\mathfrak{F}^{\mathrm{GD}}_v, \mathfrak{F}^{\mathrm{GD},-1}_v, \mathfrak{C}^{\mathrm{GD}}_v)$ is a consequence of certain special cases of Conjecture \ref{conj3.1}.

\begin{prop}\label{prop3.4}
We have the following.
\begin{enumerate}
    \item[(a)] If Conjecture \ref{conj3.1} holds for $g$ and $2n$, then we have the Fourier vanishing
    \begin{gather*}
    \mathfrak{F}_{v,2i + 1}^{\mathrm{GD}} = 0, \quad \mathfrak{F}_{v,2j + 1}^{\mathrm{GD},-1} =0 , \quad 0 \leq i, j < 2n,\\
    \mathfrak{F}_{v,2i}^{\mathrm{GD}} \circ \mathfrak{F}_{v,2j}^{\mathrm{GD},-1} =0, \quad   \mathfrak{F}_{v,2i}^{\mathrm{GD},-1} \circ \mathfrak{F}_{v,2j}^{\mathrm{GD}} =0, \quad 0 \leq i, j \leq 2n, \quad i+j \neq 2n. 
    \end{gather*}
    \item[(b)] If Conjecture \ref{conj3.1} holds for $g$ and $3n$, then the convolution class is of pure degree
    \[
    \mathfrak{C}^{\mathrm{GD}}_v \in \mathrm{CH}^{2n}( S^{[n]} \times S^{[n]} \times S^{[n]} )
    \]
    and is independent of the choice of the primitive Mukai vector $v = (r, mH, s)$ as long as $v^2=2n-2$.
\end{enumerate}
\end{prop}

\begin{proof}
   Both statements above are known in cohomology by (\ref{coh_match}) and Proposition \ref{prop1.8}. Therefore, they concern lifting certain cohomological relations to Chow; since all the relevant classes are generically defined by construction, the proposition follows from Theorem \ref{thm3.2} 
\end{proof}

Assuming Conjecture \ref{conj3.1} for $g$ and $2n$, Proposition \ref{prop3.4}(a) (together with Proposition \ref{prop1.3}) yields generically defined Fourier-stable Chow--K\"unneth decompositions
\begin{equation} \label{CK1}
h(M'_{S,v}) = \bigoplus_{k=0}^{2n} h_{2k}(M'_{S,v}), \quad
h(S^{[n]}) = \bigoplus_{i=0}^{2n} h_{2k}(S^{[n]})
\end{equation}
with projectors given by the classes (\ref{Fourier_GD}) as in (\ref{projectors}). By \cite[Proposition 1.6]{LV}, the existence and the uniqueness of generically defined Chow--K\"unneth decompositions are known for $M'_{S,v}$ and~$S^{[n]}$ if we assume the $2n$-th Franchetta property. Here the explicit forms of the projectors are needed to deduce the Fourier stability, which is crucial in the proof of Theorem~\ref{thm0.7}.

From now on, when we assume that Conjecture \ref{conj3.1} holds for $g$ and $2n$, we work with the (canonical) generically defined Chow--K\"unneth decompositions \eqref{CK1} without specifying the projectors.

Further assuming Conjecture \ref{conj3.1} for $g$ and $3n$, Proposition \ref{prop3.4}(b) (together with Proposition~\ref{prop1.6}) shows that the graded algebra object
\[
\bigoplus_{k=0}^{2n} h_{2k}(M'_{S,v}), \quad \cup: h_{2i}(M'_{S,v}) \times h_{2j}(M'_{S,v}) \to h_{2i+2j}(M'_{S,v})
\]
is isomorphic through $\widetilde{\FF}_v^{\mathrm{GD}}$ to the graded algebra object
\begin{gather*}
\bigoplus_{k=0}^{2n} h_{4n-2k}(S^{[n]})(2n-2k), \\
\ast: h_{4n-2i}(S^{[n]})(2n-2i) \times h_{4n-2j}(S^{[n]})(2n-2j) \to h_{4n-2i-2j}(S^{[n]})(2n-2i-2j)
\end{gather*}
with respect to the generically defined convolution class $\mathfrak{C}^{\mathrm{GD}}_v$.

Next, we consider a generically defined version of the Fourier autoduality in Section \ref{sec2.2} for the Hilbert scheme $S^{[n]}$. Any ample class $L \in  \mathrm{CH}^1(S^{[n]})$ of BBF norm $2d$ is obviously generically defined. Following \eqref{eq:defL} and \eqref{eq:defLambda}, we then define
\[
e_L^{\mathrm{GD}} \in \mathrm{GDCH}^{2n + 1}(S^{[n]} \times S^{[n]}), \quad f_L^{\mathrm{GD}} \in \mathrm{GDCH}^{2n - 1}(S^{[n]} \times S^{[n]}),
\]
the latter using the generically defined Fourier transforms $\mathfrak{F}_v^{\mathrm{GD}}, \mathfrak{F}_v^{\mathrm{GD}, -1}$. These in turn yield the generically defined Fourier self-transforms
\[
\FF_L^{\mathrm{GD}} := \exp(e_L^{\mathrm{GD}})\exp(-f_L^{\mathrm{GD}})\exp(e_L^{\mathrm{GD}}) \in \mathrm{GDCH}^*(S^{[n]} \times S^{[n]}), \quad \FF_L^{\mathrm{GD}, -1} \in\mathrm{GDCH}^*(S^{[n]} \times S^{[n]}),
\]
the normalizations $\widetilde{\FF}_L^{\mathrm{GD}}, \widetilde{\FF}_L^{\mathrm{GD}, -1}$, and the convolution class
\[
\mathfrak{C}_L^{\mathrm{GD}} \in \mathrm{GDCH}^*(S^{[n]} \times S^{[n]} \times S^{[n]}).
\]

The next proposition concerning $(\FF_L^{\mathrm{GD}}, \FF_L^{\mathrm{GD}, -1}, \mathfrak{C}_L^{\mathrm{GD}})$ is parallel to Proposition \ref{prop3.4}.

\begin{prop}\label{prop3.6}
    We have the following.
    \begin{enumerate}
        \item[(a)]  If Conjecture \ref{conj3.1} holds for $g$ and $2n$, then we have the Fourier vanishing
    \begin{gather*}
    \mathfrak{F}_{L,2i + 1}^{\mathrm{GD}} = 0, \quad \mathfrak{F}_{L,2j + 1}^{\mathrm{GD},-1} =0 , \quad 0 \leq i, j < 2n,\\
    \mathfrak{F}_{L,2i}^{\mathrm{GD}} \circ \mathfrak{F}_{L,2j}^{\mathrm{GD},-1} =0, \quad   \mathfrak{F}_{L,2i}^{\mathrm{GD},-1} \circ \mathfrak{F}_{L,2j}^{\mathrm{GD}} =0, \quad 0 \leq i, j \leq 2n, \quad i+j \neq 2n. 
    \end{gather*}
        \item[(b)] If Conjecture \ref{conj3.1} holds for $g$ and $3n$, then we have 
        \begin{equation*}
        \mathfrak{C}^{\mathrm{GD}}_L = \mathfrak{C}^{\mathrm{GD}}_v \in \mathrm{CH}^{2n}(S^{[n]} \times S^{[n]} \times S^{[n]}).
        \end{equation*}
    \end{enumerate}
\end{prop}

\begin{proof}
Similarly to the proof of Proposition \ref{prop3.4}, both statements are obtained by lifting relations of generically defined cycles from cohomology to Chow.
\end{proof}

Assuming Conjecture \ref{conj3.1} for $g$ and $2n$, Proposition \ref{prop3.6}(a) implies that the decomposition of $h(S^{[n]})$ in \eqref{CK1} is also stable under the generically defined Fourier self-transforms. Further assuming Conjecture \ref{conj3.1} for $g$ and $3n$, Proposition \ref{prop3.6}(b) shows that the graded algebra object
\[
\bigoplus_{k=0}^{2n} h_{2k}(S^{[n]}), \quad \cup: h_{2i}(S^{[n]}) \times h_{2j}(S^{[n]}) \to h_{2i+2j}(S^{[n]})
\]
is isomorphic through $\widetilde{\FF}_L^{\mathrm{GD}}$ to the graded algebra object
\begin{gather*}
\bigoplus_{k=0}^{2n} h_{4n-2k}(S^{[n]})(2n-2k), \\
\ast: h_{4n-2i}(S^{[n]})(2n-2i) \times h_{4n-2j}(S^{[n]})(2n-2j) \to h_{4n-2i-2j}(S^{[n]})(2n-2i-2j)
\end{gather*}
with respect to the convolution class $\mathfrak{C}^{\mathrm{GD}}_L = \mathfrak{C}^{\mathrm{GD}}_v$.

\begin{rmk}
    Since we are working with the Hilbert scheme $S^{[n]}$, another approach is to use the derived autoequivalence of $S$ induced by the ideal sheaf $\CI_\Delta$ of the diagonal $\Delta \subset S \times S$, and conjugate the BKR correspondence to get a derived autoequivalence
    \[
    \Phi_{\CI^{[n]}_{\Delta}}: D^b(S^{[n]}) \xrightarrow{\simeq} D^b(S^{[n]}).
    \]
    The kappa class of the kernel $\CI^{[n]}_{\Delta}$ is an alternative source of a generically defined Fourier self-transform. We omit the details of this approach.
\end{rmk}

\subsection{Proof of Theorem \ref{thm0.7}}\label{sec3.4}

The proof is completely parallel to the proof of Theorem \ref{thm0.5} in Section \ref{sec2.3}, where it suffices to relate the motives of $S^{[n]}$ and~$M'_{S, v}$.

We first assume that Conjecture \ref{conj3.1} holds for $g$ and $2n$. Recall the generically defined Chow--K\"unneth decompositions \eqref{CK1}. We construct an isomorphism of Chow motives 
\begin{equation*}
h_{2k}(S^{[n]}) \simeq h_{2k}(M'_{S,v})
\end{equation*}
as follows.

By Propositions \ref{prop3.4}(a) and \ref{prop3.6}(a) (together with Proposition \ref{prop1.3}(b)), we obtain mutually inverse isomorphisms of Chow motives
\begin{gather*}
h_{4n-2k}(S^{[n]}) (2n-2k)\xrightleftharpoons[~~~\widetilde{\mathfrak{F}}^{\mathrm{GD},-1}_{v, 4n-2k}~~~]{~~~\widetilde{\mathfrak{F}}^{\mathrm{GD}}_{v, 2k}~~~} h_{2k}(M'_{S,v}),
\label{cup-conv1_new} \\
h_{4n-2k}(S^{[n]}) (2n-2k) \xrightleftharpoons[~~~\mathfrak{F}^{\mathrm{GD},-1}_{L, 4n-2k}~~~]{~~~\mathfrak{F}^{\mathrm{GD}}_{L, 2k}~~~} h_{2k}(S^{[n]}).\label{cup-conv2_new}
\end{gather*}
Their composition yields mutually inverse isomorphisms of Chow motives
\[
h_{2k}(S^{[n]}) \xrightleftharpoons[~~~\mathfrak{F}^{\mathrm{GD}}_{L,2k}\circ\widetilde{\mathfrak{F}}^{\mathrm{GD},-1}_{v, 4n-2k}~~~]{ ~~~\widetilde{\mathfrak{F}}_{v, 2k}^{\mathrm{GD}}\circ \mathfrak{F}^{\mathrm{GD},-1}_{L, 4n-2k}~~~} h_{2k}(M'_{S,v})
\]
which completes the proof of Theorem \ref{thm0.7}(a).

If we further assume that Conjecture \ref{conj3.1} holds for $g$ and $3n$, then the desired isomorphism respecting the cup-product is obtained as the composition of the isomorphism
\[
\left(\bigoplus_{k = 0}^{2n}h_{2k}(M'_{S, v}), \cup\right) \simeq \left(\bigoplus_{k = 0}^{2n}h_{4n-2k}(S^{[n]})(2n-2k), \ast\right)
\]
shown in Proposition \ref{prop3.4}(b), and the isomorphism
\[
\left(\bigoplus_{k = 0}^{2n}h_{2k}(S^{[n]}), \cup\right) \simeq \left(\bigoplus_{k = 0}^{2n}h_{4n-2k}(S^{[n]})(2n-2k), \ast\right)
\]
shown in Proposition \ref{prop3.6}(b). The proof of Theorem \ref{thm0.7}(b) is also complete. \qed

\end{document}